\documentclass[a4paper,envcountsame]{llncs}
\usepackage[T1]{fontenc}
\usepackage[utf8]{inputenc}
\usepackage{graphicx}
\usepackage{amsmath}
\usepackage{a4wide}
\usepackage{graphics}
\usepackage{amssymb}
\usepackage{algorithmic}
\usepackage{algorithm}
\usepackage{multirow}
\usepackage{tikz}
\usetikzlibrary{shapes.geometric,positioning,calc,decorations,
    decorations.pathmorphing,decorations.pathreplacing,chains,scopes}
\usepackage{hyperref}

\def\cO{\ensuremath{\mathcal{O}}}
\def\fkl{\ensuremath{\mathfrak{l}}}
\def\fkp{\ensuremath{\mathfrak{p}}}

\def\fkP{\ensuremath{\mathfrak{P}}}
\def\fka{\ensuremath{\mathfrak{a}}}
\def\fkb{\ensuremath{\mathfrak{b}}}
\def\fkf{\ensuremath{\mathfrak{f}}}
\def\fku{\ensuremath{\mathfrak{u}}}
\def\fkC{\ensuremath{\mathfrak{C}}}
\def\bF{\ensuremath{\mathbb{F}}}
\def\bC{\ensuremath{\mathbb{C}}}

\def\bZ{\ensuremath{\mathbb{Z}}}
\def\bQ{\ensuremath{\mathbb{Q}}}
\newcommand{\F}{\mathbb{F}}

\newcommand{\Q}{\mathbb{Q}}

\newcommand{\Z}{\mathbb{Z}}

\newcommand{\C}{\mathbb{C}}
\newcommand{\R}{\mathbb{R}}

\def\coloneq{\mathrel{:=}}
\DeclareMathOperator{\End}{End}

\DeclareMathOperator{\Norm}{Norm}
\DeclareMathOperator{\Ker}{Ker}
\DeclareMathOperator{\lcm}{lcm}
\DeclareMathOperator{\Sp}{Sp}
\DeclareMathOperator{\Tr}{Tr}
\DeclareMathOperator{\Cl}{Cl}
\let\div\relax
\DeclareMathOperator{\div}{div}

\newcommand{\mtodo}[2][red]{\tikz[remember
picture]\coordinate(xtodo);\marginpar{%
    \begin{tikzpicture}[remember picture]
        \node[rectangle,draw=#1,rounded corners,fill=#1!20,thick] (ytodo) {%
        \begin{minipage}{\marginparwidth}\footnotesize\raggedright #2
        \end{minipage}};
        \draw[overlay,#1,opacity=0.5] (xtodo) to[bend left=15] (ytodo);
    \end{tikzpicture}}\kern0pt}
\def\RPhi{\mathrm\Phi}
\title{Isogeny graphs with maximal real multiplication}
\author{Sorina Ionica \inst{1,2} \and Emmanuel Thom\'e \inst{3}}
\institute{Institut de Mathématiques de Bordeaux \& Inria Bordeaux - Sud-Ouest\\
  200, avenue de la Vieille Tour--33405 Talence -- France
  \and
   Université de Picardie Jules Verne\\
       33 Rue Saint Leu Amiens 80039--France\\
     \and
     CARAMBA Project -- INRIA Nancy Grand Est\\
     615 rue du Jardin Botanique--54602 Villiers-les-Nancy -- France
  }

\begin{document}

\maketitle
%\footnotetext{\input{|"bash rev.sh"}}

\begin{abstract}
    An isogeny graph is a graph whose vertices are principally polarizable
    abelian varieties and whose edges are isogenies between these
    varieties.  In his thesis, Kohel describes the structure of isogeny
    graphs for elliptic curves and shows that one may compute the
    endomorphism ring of an elliptic curve defined over a finite field by
    using a depth-first search (DFS) algorithm in the graph. In dimension
    2, the structure of isogeny graphs is less understood and existing
    algorithms for computing endomorphism rings are very expensive.  In
    this article, we show that, under certain conditions, the problem
    of determining the endomorphism ring can also be solved in genus~2
    with a DFS-based algorithm. We consider the case of genus-2 Jacobians
    with complex multiplication, with the assumptions that the real
    multiplication subring has class number one and is locally maximal at $\ell$, for $\ell$ a fixed prime. We describe the isogeny graphs in that case, by considering cyclic isogenies of degree $\ell$, under the assumption that there is an ideal $\fkl$ of norm $\ell$ in $K_0$ which is generated by a totally positive algebraic integer. The resulting algorithm is implemented over finite fields, and examples are provided. To the best of our knowledge, this is the first DFS-based algorithm in genus~2.
\end{abstract}

\section{Introduction}

%\todo{La notation $\ell=\fkl_1\fkl_2$ et $\ell$ premier est à peu près
%    implicite dans l'ensemble du texte, mais ce n'est pas uniforme tout
%    le temps, il doit y avoir encore quelques endroits où on passe notre
%    temps à redéfinir\ldots}

Isogeny graphs are graphs whose vertices are simple principally
polarizable abelian varieties (p.p.a.v.) and whose edges are isogenies between
these varieties. Isogeny graphs were first studied by Kohel~\cite{Kohel},
who proves that in the case of elliptic curves, we may use these
structures to compute the endomorphism ring of an elliptic curve. Kohel
identifies three types of $\ell$-isogenies (i.e. of degree~$\ell$) in the
graph: ascending, descending and horizontal. The ascending (descending) type corresponds to the case of an isogeny between two elliptic curves, such that the
endomorphism ring of the domain (co-domain) curve is contained in the endomorphism ring of the co-domain (domain) curve. The horizontal type is that of an isogeny between two genus 1 curves with isomorphic endomorphism rings. As a consequence, computing the $\ell$-adic valuation of the conductor of the endomorphism ring can be done by a depth-first search algorithm in the isogeny graph~\cite{Kohel}.  In the case of genus-2 Jacobians, designing a similar algorithm for
endomorphism ring computation requires a good understanding of the
isogeny graph structure.
%Another obstruction is the lack of algorithms for computing
%$\ell$-isogenies over finite fields. However, there are
%algorithms~\cite{LubRob,CosRob} for computing $(\ell,\ell)$-isogenies. 

Let $K$ be a primitive quartic CM field and $K_0$ its totally real
subfield. In this paper, we study subgraphs of isogenies whose vertices
are all genus-2 Jacobians with endomorphism ring isomorphic to an order
of $K$ whose real multiplication suborder is locally maximal at $\ell$.
Furthermore, we assume that $\cO_{K_0}$ is principal, that there is a degree 1 ideal $\fkl$ lying over $\ell$ in $\cO_{K_0}$, and that this ideal is generated by a totally positive algebraic integer.

We show that the lattice of orders meeting these conditions has a simple
2-dimensional grid structure when we localize orders at $\ell$. This
results into a classification of isogenies in the isogeny graph into
three types: ascending, descending and horizontal, where these
qualificatives apply separately to the two ``dimensions'' of the lattice
of orders. Moreover, we consider $\ell$-isogenies, which are a
generalization of $\ell$-isogenies between elliptic curves to the higher
dimensional principally polarized abelian varieties (see~Definition
\ref{ell-isogeny}). We show that any $\ell$-isogeny that is such that the
two endomorphism rings contain $\cO_{K_0}$ is a composition of two
isogenies of degree $\ell$ that preserve real multiplication. As a consequence, we design a depth-first search algorithm for computing endomorphism rings in the
$\ell$-isogeny graph, based on Cosset and Robert's algorithm for
constructing $\ell$-isogenies over finite fields. To the best of
our knowledge, this is the first depth-first search algorithm for
computing locally at small prime numbers $\ell$ the endomorphism ring of
an ordinary genus-2 Jacobian. With our method, as well as with the
Eisenträger-Lauter algorithm~\cite{EisLau}, the dominant part of the
complexity is given by the computation of a subgroup of the
$\ell$-torsion. Our analysis shows that our algorithm performs faster,
since a smaller torsion subgroup is computed, defined over a smaller
field.  

This paper is organized as follows. Section~\ref{sec:background} provides
background material concerning isogeny graphs, $\cO_{K_0}$-orders of
quartic CM fields, as well as the definition and some properties of the
Tate pairing. In Section~\ref{Isogenies} we give formulae for cyclic
isogenies between principally polarized complex tori with maximal real
multiplication, and describe the structure of the graph whose edges are these isogenies.
From this, in Section~\ref{GraphStructure} we deduce the structure of the graph whose vertices are p.p.a.v. 
with maximal real multiplication, defined over finite fields, and whose edges are cyclic isogenies between
these varieties. In Section~\ref{PairingTheGraph} we show
that the computation of the Tate pairing allows us to orient ourselves in
the isogeny graph. Finally, in Section~\ref{sec:algorithm} we give our
algorithm for endomorphism ring computation when the real multiplication
is maximal, compare its performance to the one of Eisenträger and
Lauter's algorithm, and report on practical experiments over finite
fields.

\subsubsection*{Related work.}
Our work is publicly available at~\url{https://arxiv.org/abs/1407.6672}
and focuses on studying a graph structure between principally polarized
abelian varieties. For generalizations of this work to the case where the
vertices of the graph are abelian varieties with non-principal fixed
polarizations, the reader is referred on the one hand to the the more recent work of Hunter Brooks \textit{et al.}~\cite{Brooks} that takes a $p$-adic
approach to prove this graph structure. The recently defended thesis of Chloe Martindale~\cite{Martindale} also revisits this construction, using a complex-analytic approach.

We present results regarding an isomorphism between an isogeny graph
between abelian varieties defined over finite fields and the graph of
their canonical lifts (Section~\ref{GraphStructure}). To the best of our
knowledge, these results are not to be found anywhere else in the
literature. This graph isomorphism is used in several steps of the proof
developed in~\cite{Brooks} (e.g.\ proof of Proposition 5.1 in~\cite{Brooks}, as well as the remark on page 20 on that same paper,
regarding the fact that the Shimura class group action is free).

Finally, from an algorithmic point of view, \cite{Brooks} focuses on
applications that need to compute, from a given abelian variety with CM,
an isogeny path towards an abelian variety with maximal complex
multiplication. The present work proposes an algorithm for computing
endomorphism rings, via a depth first search method. To this purpose, we
present several results on the Tate pairing (see
Section~\ref{PairingTheGraph}) which are not to be found elsewhere in the literature.

\section*{Acknowledgements}

This work originates from discussions during a visit at University of
Caen in November 2011. We are indebted to John Boxall for sharing his
ideas regarding the computation of isogenies preserving the real
multiplication, and providing guidance for improving the writing of this
paper. We thank David Gruenewald, Ben Smith and Damien Robert
for helpful and inspiring discussions. Finally, we are grateful to Marco
Streng and to Gaetan Bisson for proofreading an early version of this
manuscript. % We thank the anonymous reviewers for their comments on a previous version of this manuscript. 

\section{Background and notations}
\label{sec:background}

It is well known that in the case of elliptic curves with complex
multiplication by an imaginary quadratic field $K$, the lattice of orders
of $K$ has the structure of a tower. This results in an easy way to
classify isogenies and navigate in isogeny graphs~\cite{Kohel,FouMor,IonJou2}. Throughout this paper, we are concerned with the genus~2 case.

Let then $K$ be a primitive quartic CM field, with totally real subfield $K_0$.
Principally polarized abelian surfaces considered in this paper are
assumed to be \emph{simple}, i.e. not isogenous to a product of elliptic curves
 over the algebraic closure of their field of definition. The
quartic CM field $K$ is primitive, i.e. it does not contain a totally
imaginary subfield. A CM-type ${\RPhi}$ is a pair of non-complex conjugate embeddings of $K$ in $\C$
\begin{eqnarray*}
    \RPhi(z)=\{\phi_1(z),\phi_2(z)\}.
\end{eqnarray*}
\noindent
\let\xoplus+
We assume that $K_0$ has class number one.
This implies in particular that the maximal order $\cO_K$ is a module over the principal
ideal ring $\cO_{K_0}$, whence we may define $\eta$ such that
\begin{eqnarray}\label{eta}
\cO_K=\cO_{K_0}\xoplus \cO_{K_0}\eta.
\end{eqnarray}
The notation $\eta$ will be retained throughout the paper.
For an abelian variety $A$ defined over a perfect field $F$ we denote by $\End(A)$ the endomorphism ring of $A$ over the algebraic closure of $F$ and by $\End^0(A)=\End(A)\otimes \Q$. %An abelian variety with CM by an order $\cO$ in $K$ is a pair $(A,i_A)$, where $i_A:\cO\rightarrow \End^0(A)$ is a ring embedding.
 
Several results of the article will involve a prime number $\ell$ and
also the finite field $\bF_q$ ($q=p^n$, with $p$ prime). We always implicitly
assume that $\ell$ is coprime to $p$. For an order $\cO$ in $K$ or $K_0$, we denote its localization at $\ell$ by $\cO_{\ell}=\cO\otimes \Z_{\ell}$. Note that the case which matters for our point of view is when $\ell$ splits as two distinct degree-one prime ideals $\fkl_1$ and $\fkl_2$ in $\cO_{K_0}$. How the ideals
$\fkl_{1}$ and $\fkl_{2}$ split in $\cO_K$ is not determined a priori, however.

\subsection{Isogeny graphs: definitions and terminology}
\label{GraphDefinition}

In this paper, we are interested in isogeny graphs whose nodes are all isomorphism classes of principally polarizable abelian surfaces (i.e. Jacobians of hyperelliptic genus-2 curves) with CM by $K$ and whose edges are isogenies between them, up to isomorphism. 

    \begin{definition}
        Let $I: A\rightarrow B$ be an isogeny between polarized
        abelian varieties and let $\lambda$ be a fixed polarization on $B$. The induced polarization on $A$, that we denote by $I^*\lambda$, is defined by
        \begin{eqnarray*}
            I^*\lambda=\hat{I}\circ \lambda\circ I.
       \end{eqnarray*}
      \end{definition}

\noindent
We denote by $(A,\lambda)$ a polarized abelian variety with a fixed polarization $\lambda$. We recall here the definition of an $\ell$-isogeny. 

\begin{definition}\label{ell-isogeny}
Let $I: (A,\lambda_A)\rightarrow (B,\lambda_B)$ be an isogeny between
principally polarized abelian varieties. We will say that $I$ is an $\ell$-isogeny if $I^*\lambda_B=\ell \lambda_A$.
\end{definition}
One can easily see that these isogenies have degree $\ell^2$ and have
kernel isomorphic to $\Z/\ell \Z\times \Z/\ell \Z$. The fact that for
$I: (A,\lambda_A)\rightarrow (B,\lambda_B)$, one has $I^*\lambda_B=\ell
\lambda_A$ is equivalent to $\Ker I$ being maximal isotropic with respect
to the Weil pairing, i.e. the $\ell$-Weil pairing restricts trivially to $\Ker I$ and $\Ker I$ is not properly contained in any other such subgroup (see~\cite[Prop. 13.8]{Milne}). Note that in the literature these isogenies are sometimes called $(\ell,\ell)$-isogenies (see for instance~\cite{LubRob,CosRob}). Since $\ell$-isogenies are a generalization of genus-1 $\ell$-isogenies, a natural
idea would be to consider the graph given by $\ell$-\textit{isogenies}
between principally polarized abelian surfaces. Recent developments on
the construction of $\ell$-isogenies~\cite{LubRob,CosRob} allowed us to
compute examples of isogeny graphs over finite fields, whose edges are
rational $\ell$-isogenies~\cite{AVISOGENIES}. It was noticed in this way
that the corresponding lattice of orders has a much more complicated
structure when compared to its genus-1 equivalent. Figure~\ref{llgraph}
displays an example of an $\ell$-isogeny graph. Identification of each variety to its dual, makes this graph non-oriented. The corresponding lattice of
orders contains two orders of index 3 (in the maximal order), which are
not contained one in the other. The existence of rational isogenies
between Jacobians corresponding to these two orders shows that we cannot
classify isogenies into ascending/descending and horizontal ones. This is
a major obstacle to designing a depth-first search algorithm for
computing the endomorphism ring.

\begin{figure}[h!]
\begin{center}
\includegraphics[scale=0.4,angle=90]{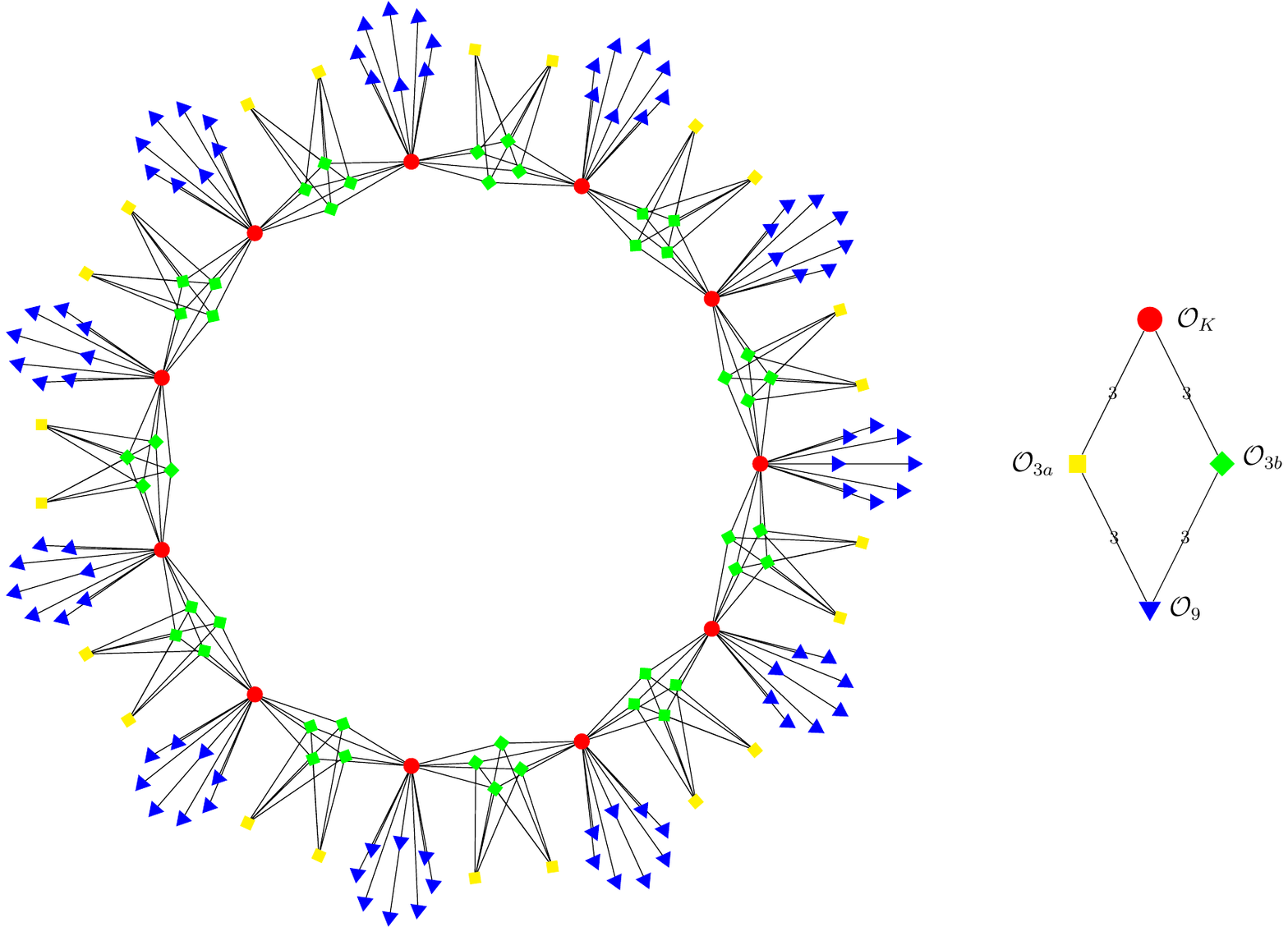}
\caption{\label{llgraph} Example of an $\ell$-isogeny graph for
    $\ell=3$ defined over a finite field $\F_p$, with $p=211$ and $K$
defined by $\alpha^4+81\alpha^2+1181$.}
\end{center}
\end{figure}

Finally, to introduce the isogeny graph of principally polarized abelian varieties with complex multiplication, we will also need the following result, which was communicated to us by Damien Robert~\cite{Robert}. 

%\mtodo{Si c'est dans BL, effectivement c'est
%inutile de dire que Damien nous l'a mentionné}

\begin{lemma}\label{LemmaDamien}
If $I:(A,\lambda_A)\rightarrow (B,\lambda_B)$ is an
isogeny between principally polarizable abelian varieties, then the
homomorphism corresponding to the induced polarization can be written as 
$I^*\lambda_B=\lambda_{A}\circ \phi$, where $\phi$ is a real
endomorphism.
\end{lemma}

We recall the definition of an abelian variety with complex (respectively real)
multiplication.

\begin{definition}
Let $A$ be a principally polarized abelian variety. Let $K$ be a
    quartic CM field, and $K_0$ its totally real subfield.
    \begin{enumerate}
        \item We say that a pair $(A,\iota)$ is an \emph{abelian variety with complex multiplication} by an order $\cO\subset K$ if there is a morphism of \bQ-algebras $\iota: K \hookrightarrow \End^0(A)$ such that $\iota^{-1}(\End(A))=\cO$
 induces a ring isomorphism between $\cO$ and $\End(A)$.
        \item Similarly, we say that a pair $(A,\iota_0)$ is an \emph{abelian variety with real multiplication} by an order $\cO_0\subset K_0$ if there is a morphism of \bQ-algebras $\iota_0: K_0 \hookrightarrow \End^0(A)$ and such that $\iota_0^{-1}(\End(A)\cap\iota_0(K_0))=\cO_0$. 
    \end{enumerate}
\end{definition} 

By the definition above, an abelian variety may have complex (resp. real)
multiplication by only one isomorphism class of orders of the lattice of orders of $K$ (resp. $K_0$).
We also note that if $(A,\iota)$ has complex multiplication by $\cO\subset K$, then $(A,\iota|_{K_0})$ has real multiplication by $\cO\cap K_0$. In this work, we fix a quartic CM field $K$ and we consider abelian varieties having complex multiplication by an order in $K$.

  %  \begin{commentmanu}
  %      
  %      Note bien qu'à moins que $A$ soit définie sur un corps fini
  %      (question à laquelle le lecteur n'a pas de réponse à ce stade),
  %      la définition qu'on prend de RM laisse ouverte la possibilité que
  %      $\End(A)$ soit de dimension seulement $g$.
  %      
  %      IOW, en toute généralité selon la définition qu'on a, RM par
  %      quelque chose n'implique pas CM par quoi que ce soit.
  %      
  %      Ce qui me fait penser
  %      qu'on devrait peut-être dire plus explicitement de quels objets
  %      on cause. Fixer un corps CM de départ et imposer, toujours, comme
  %      prérequis, d'avoir CM par ce corps ?
  %  \end{commentmanu}

\noindent
Let $A$ and $B$ be two principally polarized abelian varieties with real
(resp. complex) multiplication by an order $\cO$ and let $I:A\rightarrow B$ be a separable isogeny. We denote by $e_I$ the exponent of $I$ (i.e. the exponent of the finite group $\Ker(I)$). Since $\Ker(I) \subset A[e_I]$ (where $A[e_I]$ is the $e_I$-torsion subgroup of $A$), there is a unique isogeny $I'$ such that $II'=[e_I]$. We define the following map:

\begin{eqnarray*}
\Theta_B: \End^0(B) &\rightarrow &\End^0(A)\\
\phi &\rightarrow &\frac{1}{e_I}I'\circ \phi \circ I.\\
\end{eqnarray*}

With this in hand, we say that $I$ is an \textit{isogeny between abelian
varieties with real (resp. complex)  multiplication} by a field $K_0$
(resp. $K$) if the following diagram involving the solid arrows is commutative. (Equivalently, the diagram obtained by using $\Theta_A$ instead of $\Theta_B$ is also commutative.)%, since $\Theta_A$ and $\Theta_B$ can be extended to field isomorphisms whose composition is the identity map.)
\begin{center}
\begin{tikzpicture}
    \node (ellO) at (0,2){$K_0 (\textrm{resp.}~K)$}; 
    \node (B) at (3,2)  {$\End^0(B)$}; 
    \node (A) at (3,0) {$\End^0(A)$};
    \draw[->] (ellO) to node[above] {\small{$\iota_B$}} (B);
    \draw[->] (ellO) to node[left] {\small{$\iota_A$}} (A);
    \draw[->,bend right=28] (B) to node[left] {\small{$\Theta_{B}$}} (A);
    \draw[<-,bend left=28, dashed] (B) to node[right] {\small{$\Theta_{A}$}} (A);
\end{tikzpicture}
\end{center}

%Assume that $e_I$ equals $\ell$. This implies that there is an embedding $i_{B,\Theta}:\ell\End^0(B)\hookrightarrow \End^0(A)$. In particular, if $B$ has real multiplication by an order~$\cO$ and $i_B:\cO\rightarrow \End(B)$, this gives an embedding of $\ell\cO$ in $\End(A)$. In a symmetric way, we obtain $i_{A,\Theta}:\ell\End(A)\hookrightarrow \End(B)$.
%\textcolor{blue}{C'est pas l'inverse, plutôt ?}\textcolor{red}{J'ai
%corrige}
%\textcolor{blue}{ok. on exploite le fait qu'on est dans le cas ordinaire:
%$\End(A)$, étant inclus dans un corps CM, est commutatif}.

%\textcolor{blue}{pourquoi
%définir "préserve RM" pour des isogénies entre PAV seulement si on veut
%l'évoquer aussi pour des isogénies entre AV sans polarisation fixée ?
%}

In this work, we denote by $(A,\lambda,\iota)$ a principally polarized
abelian variety with complex multiplication.  
Note that here and throughout the paper, we shall only
distinguish isogenies up to isomorphism, regarding isogenies $I_1:A\rightarrow B$ and $I_2:A \rightarrow B$ as equivalent if $I_1=i_2\circ I_2 \circ i_1$ for any automorphisms $i_1:A\rightarrow B$ and $i_2:A\rightarrow B$. The approach we will take here is to consider the graph of \textit{all} (equivalence classes of) isogenies between principally polarizable abelian surfaces and decompose it into subgraphs whose vertices are abelian surfaces with real multiplication by a fixed order $\cO$ of $K_0$.

\begin{definition}\label{def:isogeny-graph}
Let $F$ be a perfect field. An isogeny graph of principally polarized abelian varieties defined over $F$ is a graph such that:
\begin{enumerate}
\item The vertices are isomorphism classes of principally polarized abelian varieties $(A,\lambda,\iota)$ with complex multiplication.
\item There is an edge between two classes $(A,\lambda_A,\iota_A)$ and $(B,\lambda_B,\iota_B)$ whenever there is an isogeny $I:A\rightarrow B$ between abelian varieties with CM, such that $\lambda_BI^*= \lambda_{A}\circ \phi$, for $\phi$ some real endomorphism. 
\end{enumerate}
\end{definition}

\begin{definition}
With the notation above, let $(A,\iota_A)$ and $(B,\iota_B)$ two abelian varieties with complex multiplication by a CM field $K$ and $I:A\rightarrow B$ an isogeny between them. We say that $I$ preserves real multiplication by an order \cO\ in $K_0$ if both $A$ and $B$ have real multiplication by \cO.
\end{definition}

As a consequence, for an order $\cO$ in $K_0$, we call \textit{$\cO$-layer in the graph given by Definition~\ref{def:isogeny-graph}} the subgraph whose vertices are all equivalence classes of p.p.a.v with RM by $\cO$ and whose edges are isogenies preserving real multiplication by $\cO$.
Understanding the structure of the graph then comes down to explaining
the structure of each layer and in a later step classifying isogenies beween two vertices lying at different layers of the graph. 

%\textcolor{blue}{cyclic isogenies of degree $\ell$ between principally polarized abelian varieties  n'existe pas}

In this paper, we fully describe the structure of the $\cO_{K_0}$-layer. Working towards this goal, we first identify cyclic isogenies of degree $\ell$ between principally polarizable abelian varieties with maximal real multiplication. We will show in Section~\ref{Isogenies} that a sufficient condition to guarantee the existence of isogenies of degree $\ell$ between principally polarized abelian varieties is that there is a principal ideal in $\cO_{K_0}$ of degree 1 and norm $\ell$, whose generator is totally positive. 

As a consequence, we chose to focus on the case where $\ell$ splits in $\cO_{K_0}$ into two principal ideals. Under these restrictions, we describe the simple and interesting structure of the graph of cyclic isogenies, which fits into the ascending/descending and horizontal framework. Using this graph structure, we characterize all isogenies between principally polarized abelian surfaces which preserve maximal real multiplication. This leads in particular to viewing Figure~\ref{llgraph} as derived from a more structured graph, whose characteristics are well explained.

The case when $\ell$ is ramified the graph structure is similar, as
explained in Section~\ref{Isogenies} (Remark~\ref{ramifiedCase}). In
the case of $\ell$ inert, one can see easily from Lemma~\ref{LemmaDamien} that there are no degree~$\ell$ isogenies between principally polarizable abelian varieties with CM by $K$, preserving real multiplication. Indeed, if there were, this would imply the existence of a norm $\ell$ element $\alpha\in \cO_{K_0}$. We chose not to treat the case of $\ell$ inert in this work.

%\textcolor{blue}{Une isogénie \emph{entre PAV}, donc telle que
%$I^*E_2=E_1$, ne peut pas être de
%degré $\ell$ de toute façon}

\subsection{The lattice of \texorpdfstring{$\cO_{K_0}$}{OK0}-orders in a
quartic CM field \texorpdfstring{$K$}{K}}\label{subsec:LatticeOfOrders}

A major obstacle to explaining the structure of genus~2
isogeny graphs is that the lattice of orders of $K$ lacks a concise
description. Given an isogeny $I:A \rightarrow B$ between two
abelian surfaces with degree~$\ell$, the corresponding endomorphism rings
are such that $\ell \cO_{A}\subset \cO_{B}$ and $\ell \cO_{B}\subset \cO_{A}$. Hence, even if a inclusion relation is guaranteed $\cO_{B}\subset \cO_{A}$, the index of one order in the other is bounded by $\ell^3$. Since the $\Z$-rank of orders is 4, there could be several suborders of $\cO_{A}$ with the same
index.

In this paper, we study the structure of the isogeny graph between
abelian varieties with maximal real multiplication. The first step in
this direction is to describe the structure of the lattice of orders of
$K$ which contain $\cO_{K_0}$. Following~\cite{GorLau}, we call such an
order an $\cO_{K_0}$-order. We study the conductors of such orders. We
recall that the conductor of an order~$\cO$ is the ideal
\begin{eqnarray*}
\fkf_{\cO}=\{x\in \cO_K \mid x\cO_K\subset \cO\}.
\end{eqnarray*}
\begin{lemma}\label{Goren++}
Let $K$ be a quartic CM-field and $K_0$ its real
    multiplication subfield. Assume that the class number of $K_0$ is 1.
    Then the following hold:
\begin{enumerate}

\item Given $\alpha\in \cO_{K_0}$, $\cO=\cO_{K_0}[\alpha\eta]$ is an $\cO_{K_0}$-order of conductor $\alpha \cO_{K_0}$.

\item For any $\cO_{K_0}$-order $\cO$ of $K$ there is $\alpha\in \cO_{K_0}$, $\alpha \neq 0$ such that $\cO=\cO_{K_0}[\alpha\eta]$. The element~$\alpha$ is unique up to units of $\cO_{K_0}$. 
   
\end{enumerate}
\end{lemma}
\begin{proof}
Statements~1 and~2 were given by Goren and Lauter~\cite{GorLau}, and
    characterize $\cO_{K_0}$-orders completely in our case. 
    
\end{proof}

As a consequence, we get the following result. 
\begin{lemma}
Any $\cO_{K_0}$-order is a Gorenstein order.
\end{lemma}

\begin{proof}
This is a consequence of the fact that $\cO$ is monogenic over $\cO_{K_0}$, hence the argument of~\cite[Example 2.8 and Prop. 2.7]{BuchmannLenstra} applies.\qed
\end{proof}

A first consequence of Lemma~\ref{Goren++} is that there is a bijection
between $\cO_{K_0}$-orders and principal ideals in $\cO_{K_0}$, which
associates to every order the ideal $\fkf\cap\cO_{K_0}$. For
brevity we still call the latter the conductor and denote it by~$\fkf$.

Using the particular form of $\cO_K$ as a monogenic $\cO_{K_0}$-module,
we may rewrite the conductor differently. For a fixed element $\omega
\in \cO_K$, we define the conductor of \cO\ with respect to~$\omega$ to
be the ideal
\begin{eqnarray*}
\fkf_{\omega,\cO}=\{x\in \cO_K \mid x\omega \in \cO\}.
\end{eqnarray*}
The following statement is an immediate consequence of
Lemma~\ref{Goren++}.

\begin{lemma}\label{secondConductor} 
For any $\cO_{K_0}$-order $\cO$ and any~$\eta$ such that
$\cO_K=\cO_{K_0}[\eta]$, we have $\fkf_{\cO}=\fkf_{{\eta},\cO}$.
\end{lemma}
\medskip
Now let $\cO$ be an order in $K$ with locally maximal real multiplication at $\ell$ (i.e. $(\cO_{K_0})_{\ell}\subset \cO_{\ell}$). Assume that the index of $\cO$ is divisible by a power of $\ell$ and that $\ell$ splits in $\cO_{K_0}$ and let $\ell=\fkl_1\fkl_2$. Then $\cO_{\ell}$ is isomorphic to the localization of a $\cO_{K_0}$-order, whose conductor $\fkf$ has a unique factorization into prime ideals containing $\fkl_1^{e_1}\fkl_2^{e_2}$. Locally at~$\ell$, the lattice of orders of index divisible by~$\ell$ has the form given in Figure~\ref{fig:order-lattice}. This is equivalent to the following statement.

%\sorina{Ici j'ai change la formulation locally at \ell}

\begin{lemma}\label{lemma:ordlat-2dim}
Let $\cO$ be an order in $K$, with locally maximal real multiplication. The position of~$\cO_{\ell}$ within the lattice of $\cO_{K_0}$-orders localized at $\ell$ is given by the valuations $\nu_{\fkl_i}(\fkf_\cO)$, for $i=1,2$.
\end{lemma}

\begin{figure}
\begin{center}
    \begin{tikzpicture}[
            point/.style={circle, fill=black,
                minimum size=3pt, inner sep=0pt},
            mu1/.style={orange, thick},
            mu2/.style={violet, thick},
            % mu1/.style={thick},
            % mu2/.style={thick},
        ]
        \node[point] (A00) at (0,0) {};
        \node[point,below  left=20pt of A00] (A01) {};
        \node[point,below right=20pt of A00] (A10) {};
        \node[point,below  left=20pt of A01] (A02) {};
        \node[point,below  left=20pt of A10] (A11) {};
        \node[point,below right=20pt of A10] (A20) {};
        \node[point,below  left=20pt of A02] (A03) {};
        \node[point,below  left=20pt of A11] (A12) {};
        \node[point,below  left=20pt of A20] (A21) {};
        \node[point,below right=20pt of A20] (A30) {};
        \path[use as bounding box] (A03) rectangle (A30 |- A00);
        \node[right=20pt of A30] (right) {};
    \draw[gray,dashed] (A00) -- (A00 -| right) node[right]
    {\normalcolor$\nu_l(\Norm(\fkf_\cO))=0$};
    \draw[gray,dashed] (A10) -- (A10 -| right) node[right]
    {\normalcolor$\nu_l(\Norm(\fkf_\cO))=1$};
    \draw[gray,dashed] (A20) -- (A20 -| right) node[right]
    {\normalcolor$\nu_l(\Norm(\fkf_\cO))=2$};
    \draw[gray,dashed] (A30) -- (A30 -| right) node[right]
    {\normalcolor$\nu_l(\Norm(\fkf_\cO))=3$};
        \draw[mu1,-] (A00) -- node[above left] {$\fkl_1$} (A01) -- (A02);
        \draw[mu1,-] (A10) -- (A11);
        \draw[mu2,-] (A00) -- node[above right] {$\fkl_2$} (A10) -- (A20);
        \draw[mu2,-] (A01) -- (A11);
        \draw[mu1,dashed,-] (A02) -- (A03);
        \draw[mu1,dashed,-] (A11) -- (A12);
        \draw[mu1,dashed,-] (A20) -- (A21);
        \draw[mu2,dashed,-] (A20) -- (A30);
        \draw[mu2,dashed,-] (A11) -- (A21);
        \draw[mu2,dashed,-] (A02) -- (A12);
    \end{tikzpicture}
    \caption{\label{fig:order-lattice}The lattice of orders}
\end{center}
\end{figure}
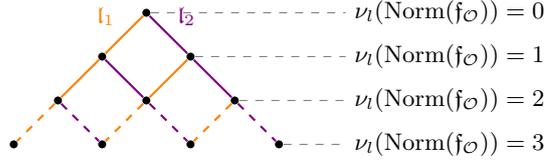

\noindent
We call level in the lattice of orders the set of all
orders having the same $\ell$-adic valuation of the norm of the
conductor. For example, level 2 in Figure~\ref{fig:order-lattice} is
formed by three orders with conductors $\fkl_1^2,\fkl_1\fkl_2$ and
$\fkl_2^2$, respectively. This distribution of orders on levels leads to a classification of isogenies into descending and ascending ones, which is the key point to a DFS algorithm for navigating in the isogeny graph, just like in the elliptic curve case. This will be furthered detailed in Section~\ref{GraphStructure}.

%\begin{commentmanu}
%    On définit, plus loin, la notion de $\fkl_i$-isogénie
%    \emph{horizontale}/\emph{ascendante}/\emph{descendante}. Cette notion
%    est appréciée en fonction de ce qui se passe modulo le \fkl\ qu'on
%    regarde, pas modulo l'autre! Tandis que la $\ell$-valuation du
%    conducteur, ça mélange les deux.

%    Bon d'accord, on passe l'essentiel de notre temps à causer des
%    isogénies cycliques. Mais je préférerais que le caractère
%    \emph{horizontal}/\emph{ascendant}/\emph{descendant} d'une
%    \fkl-isogénie soit apprécié plutôt par ce qu'elle fait à la $\fkl$-valuation (et non la
%    $\ell$-valuation) du conducteur de $\End(A)$.
%\end{commentmanu}

\subsection{The Tate pairing}\label{subsec:TatePairing}

Let $A$ be a polarized abelian surface, defined over a perfect 
field $F$. We denote by $A[m]$ the $m$-torsion subgroup. We denote by
$\mu_m$ the group of $m$-th roots of unity and by 
\begin{eqnarray*}
W_m:A[m]\times \hat{A}[m] &\rightarrow &\mu_m
\end{eqnarray*}
the $m$-Weil pairing on the abelian surface.  

In this paper, we are only interested in the Tate pairing over finite
fields. We give a specialized definition of the pairing in this case,
following~\cite{Schmoyer,Ionica}. More precisely, let $F=\F_q$ and suppose that we have $m\mathbin|\#A(\F_q)$. We denote by $k$ the \textit{embedding degree with respect to} $m$, i.e. the smallest integer $k\geq 0$ such that
$m\mathbin|q^k-1$. Moreover, we assume that $A[m]$ is
defined over $\F_{q^k}$. We define the Tate pairing as
$$T_m(\cdot,\cdot):\left\{
\begin{array}{rcl}
A(\F_{q^k})/mA(\F_{q^k})\times \hat{A}[m](\F_{q^k})&\rightarrow &\mu_m\\
(P,Q) &\mapsto &  W_m(\pi_k(\bar{P})-\bar{P},Q),
\end{array}\right.$$
where $\pi_{k}$ is $k$-th power of the Frobenius
endomorphism of $A$ and $\bar{P}$ is any point such that $m\bar{P}=P$.
Note that since $A[m]\subseteq A(\F_{q^k})$, this definition is
independent of the choice of $\bar{P}$.  Indeed, if $\bar{P}_1$ is a
second point such that $m\bar{P}_1=P$, then $\bar{P}_1=\bar{P}+T$, where
$T$ is a $m$-torsion point, and $\pi_k(\bar{P}_1)-\bar{P}_1=
\pi_k(\bar{P})-\bar{P}$. 

\noindent
For a fixed polarization $\lambda:A\rightarrow \hat{A}$ we
define a pairing on $A$ itself
$$
T_m^{\lambda}(\cdot,\cdot):\left\{
\begin{array}{rcl}
A(\F_{q^k})/mA(\F_{q^k})\times A[m](\F_{q^k})&\rightarrow &\mu_m\\
(P,Q) &\mapsto & t_m(P,\lambda(Q)).
\end{array}\right.$$

If $A$ has a distinguished principal polarization and there is no risk of
confusion, we write simply $T_m(\cdot,\cdot)$ instead of
$T_m^{\lambda}(\cdot,\cdot)$.
\medskip

Lichtenbaum~\cite{Lichtenbaum} describes a version of the Tate pairing on
Jacobian varieties. Since we use Lichtenbaum's formula for computations,
we briefly recall it here. Let $D_1\in A(\F_{q^k})$ and $D_2\in
A[m](\F_{q^k})$ be two divisor classes, represented by two divisors such
that $\mathrm{supp}(D_1)\cap \mathrm{supp}(D_2)=\emptyset$. Since $D_2$
has order $m$, there is a function $f_{m,D_2}$ such that
$\div(f_{m,D_2})=mD_2$. The Lichtenbaum pairing of the divisor classes
$D_1 $ and $D_2$ is computed as
\begin{eqnarray*}
t_m(D_1,D_2)=f_{m,D_2}(D_1).
\end{eqnarray*}
The output of this pairing is defined up to a coset of $(\F_{q^k}^*)^m$.
Given that $\F_{q^k}^*/(\F_{q^k}^*)^m\simeq \mu_m$, we obtain a pairing defined as
\begin{eqnarray*}
T_m(\cdot,\cdot): A(\F_{q^k})/mA(\F_{q^k})\times A[m](\F_{q^k})
    &\rightarrow &\mu_m\\
(P,Q) &\rightarrow & t_m(P,Q)^{(q^k-1)/m}.
\end{eqnarray*}
The function $f_{m,D_2}(D_1)$ is computed using Miller's
algorithm~\cite{Miller} in $O(\log m)$ operations in $\F_{q^k}$.

\section{Isogenies preserving real multiplication}\label{Isogenies}
%\subsection{Abelian varieties with complex multiplication}

Let $K$ be a quartic CM field and $\RPhi=(\phi_1,\phi_2)$ be a CM-type.
The notation $\cO^{\dagger}$ denotes the complementary module of an order $\cO$, i.e.
$\cO^{\dagger}=\{\alpha \in K| \Tr_{K/\Q}(\alpha\cO)\subseteq \Z\}$. In this Section all abelian varieties are defined over $\C$, unless specifically stated otherwise. A principally polarized abelian surface over $\C$ with complex multiplication by an order $\cO\subset K$ is of the form $A=\C^2/\RPhi(\fka)$, where $\fka$ is a fractional ideal of $\cO$ and such that
\begin{eqnarray}\label{def:ppavideal}
\xi \fka\bar{\fka}=\cO^{\dagger}, 
\end{eqnarray}
with $\xi$ purely imaginary such that $\phi_i(\xi)$ lies on the positive imaginary axis for $i\in \{1,2\}$. The variety given by $(\fka, \xi)$ is said to be of CM-type $(\cO,\RPhi)$. The imaginary part of any Riemann form on $\C^2/\RPhi(\fka)$ writes as
%\begin{eqnarray*}
%H_{\xi}(z,w)=\sum_ {r=1}^2\frac{\xi^{\phi_r}z^{\phi_r}\bar{w}^{\phi_r}}{\Im(\tau^{\phi_r})},
%\end{eqnarray*}
\begin{eqnarray*}
E_{\xi}(z,w)=\sum _{r=1}^2 \xi^{\phi_r}(x'^{\phi_r}y^{\phi_r}-x^{\phi_r}y'^{\phi_r}),
\end{eqnarray*}
with $z=x+y\tau, w=x'+y'\tau$, where $x,y,x',y'\in \R$.

\noindent
This defines a principal polarization~\cite{Birkenhake} that we denote by $\lambda_{\xi}$. By extension, for a given isogeny $I: (\C^2/\Phi(\fka),\lambda_{\xi})\rightarrow (\C^2/\Phi(\fkb),\lambda_{\xi'})$, we call induced polarization $I^*E_{\xi'}(u,v)=E_{\xi'}(I(u),I(v))$, for all $u,v\in \C^2$. 

\noindent
Recall that we focus on the case where $\cO_{K_0}\subset\cO$.  
\begin{lemma}
Let $K$ be a quartic CM field and $K_0$ its maximal real subfield with class number 1 and let $\delta$ be the generator of $(\cO_{K_0}^{\dagger})^{-1}$ and $\mu$ a generator for the conductor of $\cO$. For every p.p.a.v. of CM-type $(\cO,\RPhi)$ given by $(\fka,\xi)$ there exists $\tau\in K$ such that 
\begin{enumerate}
\item $\fka=\cO_{K_0}+\cO_{K_0}\tau$
\item $\xi=-1/\delta\mu(\tau-\bar{\tau})$ 
\item $(\tau^{\phi_1},\tau^{\phi_2})\in \mathbb{H}_1\times (\C \setminus \R)$, where $\mathbb{H}_1$ is the upper-half plane.
\end{enumerate}
\end{lemma}

\begin{proof}
Since $\cO_{K_0}$ is a Dedekind domain and the ideal $\fka$ is an $\cO_{K_0}$-module, we may then write it as $\fka=\Lambda_1\alpha\xoplus \Lambda_2\beta$, with $\alpha,\beta \in K$, and $\Lambda_{1,2}$ two $\cO_{K_0}$-ideals. Hence we have $A\cong \C^2/\RPhi(\Lambda)$ and $\Lambda=\alpha^{-1}\fka=\Lambda_1\xoplus \Lambda_2\tau$, with $\Lambda_1$ and $\Lambda_2$ lattices in $K_0$ and $\tau=\frac\beta\alpha \in K$. Note that since $K_0$ has class number one, it follows that we can
choose $\Lambda_1=\Lambda_2=\cO_{K_0}$. 

For the proof of 2, we use the computations in~\cite[Prop. 4.2]{Spallek} which shows that
\begin{eqnarray*}
\fka\bar{\fka}=(\frac{\tau-\bar{\tau}}{\eta-\bar{\eta}}),
\end{eqnarray*}
where $\eta$ is the one defined by Equation~\eqref{eta}. 
Note that $\cO^{\dagger}=\{\alpha \in K| \Tr_{K/\Q}(\alpha\cO)\subseteq \Z\}=\{\alpha \in K| \Tr_{K/K_0}(\alpha\cO)\subseteq \cO_{K_0}\}\cO_{K_0}^{\dagger}$. From this and by using~\cite[Example 2.8]{BuchmannLenstra}, we get that $\cO^{\dagger}=\frac{1}{\mu(\eta-\bar{\eta})\delta}\cO$. We conclude that  $\xi=-1/\delta\mu(\tau-\bar{\tau})$ .

To prove 3, we look at the equality obtained in 2. The fact that $(\tau^{\phi_1},\tau^{\phi_2})\in \mathbb{H}_1\times (\C\setminus \R)$ follows from the fact that $\xi^{\phi_i}$, $i=1,2$ is on the positive imaginary axis and that we may assume, without restricting the generality, that $(\delta \mu)^{\phi_1}$ is positive.  
\qed
\end{proof}

The isogenies discussed by the following proposition were brought to our
attention by John Boxall.

\begin{proposition}
\label{prop:john}
    Let $K$ and $K_0$ be as previously stated. Let $\ell$ be a prime, and
    $\fkl\subset\cO_{K_0}$ a prime $\cO_{K_0}$-ideal of norm~$\ell$.  Let
    $A=\C^2/\Phi(\Lambda)$ be an abelian surface over \bC\ with complex
    multiplication by an $\cO_{K_0}$-order $\cO\subset K$, with
    $\Lambda=\Lambda_1+\Lambda_2\tau$. A set of representatives of the
    cyclic subgroups of $(\fkl^{-1}\Lambda)/\Lambda$, and more precisely of
    the isogenies on $A$ having these subgroup as kernels is given by
    $\{I_\infty\}\cup\{I_\rho,\
    \rho\in{\Lambda_1\Lambda_2^{-1}}/{\fkl\Lambda_1\Lambda_2^{-1}}\}$,
    where:
\begin{equation}
\label{John}
    I_\infty:
    \left\{
    \begin{array}{rcl}
        \displaystyle
A&\rightarrow&
        \displaystyle
        \C^2/\RPhi(\fkl^{-1}\Lambda_1\xoplus  \Lambda_2\tau),\\
z &\mapsto & z,
    \end{array}
    \right.
    \quad
    I_\rho:
    \left\{
    \begin{array}{rcl}
        \displaystyle
    A&\rightarrow&
        \displaystyle
        \C^2/\RPhi(\Lambda_1\xoplus 
        \fkl^{-1}\Lambda_2(\tau+\rho)),\\
z &\mapsto & z.
    \end{array}
    \right.
\end{equation}
\end{proposition}

\begin{proof}
    Our hypotheses imply that $\Lambda$ is an $\cO_{K_0}$-module of rank
    two, from which it follows that $(\fkl^{-1}\Lambda)/\Lambda$ is
    isomorphic to $(\bZ/\ell\bZ)^2$. The $\ell+1$ cyclic subgroups of
    $(\fkl^{-1}\Lambda)/\Lambda$ are the kernels of
    the isogenies given in the Proposition.\qed
\end{proof}

The isogenies given by Equation~\eqref{John} are examples of \fkl-isogenies, that we define as follows.

\begin{definition}
\label{def:lfrak-isog}
Let $\fkl$ be an prime ideal of $\cO_{K_0}$ of norm a prime number $\ell$. Then the $\fkl$-torsion of an abelian variety $A$ defined over a perfect field $F$ with real multiplication by
$\cO_{K_0}$ is given by
\begin{eqnarray*}
A[\fkl]=\{x \in A(\overline{F})~s.t.~\alpha x=0,~\forall \alpha\in \fkl\}.
\end{eqnarray*}
Isogenies with kernel a cyclic subgroup of $A[\fkl]$ of order $\ell$ are
called \emph{\fkl-isogenies}.
\end{definition}
For the commonly encountered case where $\fkl=\alpha\cO_{K_0}$ for some generator $\alpha\in\cO_{K_0}$ (which occurs in our setting since $\cO_{K_0}$ is assumed principal), the notation $A[\fkl]$ above matches with the notation $A[\alpha]$ 
representing the kernel of the endomorphism represented by $\alpha$.
In this situation, the cyclic isogenies introduced in Definition~\ref{def:lfrak-isog} are also called $\alpha$-isogenies.

In the remainder of this paper, we assume that $\fkl$ is a prime ideal of norm $\ell$. 
This implies that $\ell$ is either split or ramified in $K_0$. In this paper, we deliberately chose to focus on the split case. This restriction allows us to further design an algorithm for endomorphism ring computation, as we will explain in Section~\ref{sec:algorithm}.

Given Definition~\ref{def:lfrak-isog}, Proposition~\ref{prop:john} can be
regarded as giving formulae for a set of representatives for isomorphism classes of $\fkl$-isogenies over the complex numbers.

\medskip

The following trivial observation that \fkl-isogenies preserve the
maximal real multiplication follows directly from
$\End(\fkl^{-1}\Lambda_i)=\End(\Lambda_i)$. Later in this article we will show that a converse to this statement also holds: an isogeny which preserves the maximal real multiplication is an $\fkl$-isogeny (Proposition~\ref{allRMIsogenies}).

\begin{proposition}
\label{prop:ell-isog-preserve-rm}
Let $A$ be an abelian surface defined over $\C$ with $\End(A)$ an $\cO_{K_0}$-order. Let $I:A\rightarrow B$ be an \fkl-isogeny. Then $\End(B)$ is also an $\cO_{K_0}$-order.
\end{proposition}

The following proposition shows how polarizations can be transported
through $\fkl$-isogenies. We use here the fact that $\cO_{K_0}$ is assumed
to have class number one.

\begin{proposition}
    \label{alphapolarization}
    Let $A  $ be an abelian surface with $\End(A)$ an $\cO_{K_0}$-order.
    Let $I:A  \rightarrow B  $ be an \fkl-isogeny (following the notations
    of Proposition~\ref{prop:john}).
    Let $E_\xi$ define a principal polarization of $A$.
    If $\fkl=(\alpha)$ with $\alpha\in K_0$ totally positive,
            $E_{\alpha\xi}$ defines a principal polarization on $B$.
           % \todo{Euh, si $\alpha$ n'est pas totalement positif, le
           % problème qui se pose 
           % c'est plutôt sur l'aspect que $H_{\alpha\xi}$ ne
           % sera pas définie positive, n'est-ce pas ? Du coup c'est pas une
           % polarisation ?}
        %\item $I^*E_{\ell\xi}=\ell E_\xi$. However, $E_{\ell\xi}$
        %    does not define a principal polarization on $A_2$.
       Moreover, $I^*E_{\alpha\xi}=\alpha E_\xi$. %However, $E_{\ell\xi}$
        %    does not define a principal polarization on $A_2$.
\end{proposition}

\begin{proof}
    We follow notations of Proposition~\ref{prop:john} and take
    $I=I_{\infty}$ as an example (the other cases are similar). We can write
\begin{align*}
    E_{\xi}(x+y\tau, x'+y'\tau)
    &=E_{\alpha\xi}(\frac{x}{\alpha}+y\tau,\frac{x'}{\alpha}+y'\tau).\\
\end{align*}
Hence if $E_{\xi}$ defines a principal polarization on
$\C^2/\RPhi(\Lambda_1\xoplus \Lambda_2\tau)$ and $\alpha$ is totally
positive then $E_{\alpha\xi}$
defines principal polarizations on the variety
    $\C^2/\RPhi(\frac{\Lambda_1}{\alpha}\xoplus \Lambda_2\tau)$ (we just
    showed that the matrices of the corresponding Riemann forms are
    equal).

The fact that $I^*E_{\alpha\xi}=\alpha E_\xi$ follows from the definition
    of $I^*$, and of the Riemann forms $E_\xi$ and $E_{\alpha\xi}$.
\qed    
\end{proof}

\begin{lemma}\label{dualisogeny}
Let $A$ be a principally polarized abelian variety under the assumptions in Proposition~\ref{alphapolarization}.
The dual of an $\alpha$-isogeny starting from $A$ is an $\alpha$-isogeny.
\end{lemma}
\begin{proof}
  This follows trivially from $I^*E_{\alpha\xi}=\alpha E_\xi$, since this implies that
  $\alpha\lambda_{\xi}=\hat{I}\circ\lambda_{\alpha \xi}\circ I$, where $\lambda_{\xi}$ and
  and $\lambda_{\alpha\xi}$ are the isogenies corresponding to polarizations $E_{\xi}$ and
  $E_{\alpha \xi}$. 
\qed  
\end{proof}

%If moreover, $I$ is of degree $\ell$, the endomorphism $\phi$ corresponds to a totally positive element $\alpha$ of norm $\ell$ in $\cO_{K_0}$.

\noindent
In the remainder of this paper, we assume that $\alpha$ as in Proposition~\ref{alphapolarization} exists and is totally positive. It becomes clear then that the $\alpha$-isogenies we introduced are edges in the graph given by Definition~\ref{def:isogeny-graph}. By Proposition~\ref{prop:ell-isog-preserve-rm} they are edges in the $\cO_{K_0}$-layer of this graph.

\begin{remark}\label{polarization}
% \textcolor{blue}{Pour moi cette remarque ne sert à rien}\\
% \textcolor{red}{
If $\ell$ is a prime number such that $\ell\cO_{K_0}=\fkl_1\fkl_2$, we denote by $\alpha_i$, $i=\{1,2\}$, elements of $\cO_{K_0}$ such that $\fkl_i=\alpha_i\cO_{K_0}$. Let $I: A\rightarrow B$ be an $\fkl_1$-isogeny. Proposition~\ref{alphapolarization} implies that for a given polarization $\xi$ on $A$, $\ell \lambda_{\xi}=\hat{I}\circ (\alpha_2\lambda_{\alpha_1\xi}) \circ I$. 
% }
\end{remark}

\vspace{0.5 cm}
\noindent
Note that if $\ell$ is such that $\ell\cO_{K_0}=\fkl_1\fkl_2$, with
$\fkl_1+\fkl_2=(1)$, then the factorization of $\ell$ yields a symplectic
basis for the $\ell$-torsion. Indeed, we have $J[\ell]=J[\fkl_1]\xoplus
J[\fkl_2]$, and the following proposition establishes the symplectic
property.

\begin{proposition}
\label{prop:weil-isotropic-l1l2}
Let $J$ be a principally polarized abelian surface defined over a
number field $L$. With the notations above, we have
$W_{\ell}(P_1,P_2)=1$ for any $P_1\in J[\fkl_1]$ and $P_2\in J[\fkl_2]$.
\end{proposition}
\begin{proof}
    This can be easily checked on the complex torus
    $\C^2/\RPhi(\Lambda_1\xoplus\Lambda_2\tau)$. Let
    $P_1=\frac{x_1}{\alpha_1}+\frac{x_2}{\alpha_1}\tau \in J[\alpha_1]$
    and $P_2=\frac{y_1}{\alpha_2}+\frac{y_2}{\alpha_2}\tau \in
    J[\alpha_2]$, where $x_1,y_1\in \Lambda_1$ and $x_2,y_2\in
    \Lambda_2$.
    Then $W_{\ell}(P_1,P_2)=\exp(-2\pi i
    \ell\frac{E_{\xi}(x_1+x_2\tau,y_1+y_2\tau)}{\ell}))=1$.
\qed    
\end{proof}

The following lemma allows us to count the number of principally polarized abelian varieties with CM by an $\cO_{K_0}$-order $\cO$.. Along the lines of its proof, we also show that the action of Shimura class group of $\cO$ on the set of p.p.a.v. with CM by $\cO$ is simple and transitive.  

\begin{lemma}\label{ShimuraCardinality}
Let $\cO$ be an $\cO_{K_0}$-order in a CM quartic field $K$. The number of isomorphism classes of principally polarized abelian surfaces with CM type $(\cO,\Phi)$ is 
\begin{eqnarray*}
\frac{\#\Cl(\cO)}{\#\Cl^+(\cO_{K_0})}\cdot \#(\cO_{K_0}^{\times})^+/N_{K/K_0}(\cO^{\times}).
\end{eqnarray*}
\end{lemma}
\begin{proof}
Note first that when $\cO$ is an $\cO_{K_0}$-order, then
all principally polarized abelian varieties over $\C$ are of the
form $\C^2/\RPhi(\fka)$, with $\fka$ an invertible ideal of $\cO$. This
follows from Equation~\eqref{def:ppavideal} because $\cO$ is
Gorenstein (see Lemma~\ref{Goren++}), and thus $\cO^\dagger$ is
invertible. Just like in the case of $\cO=\cO_K$ treated by Shimura, it then follows that there is a transitive simple action of the Shimura class group of $\cO$, denoted by $\mathfrak{C}(\cO)$, on the set of principally polarized abelian varieties with CM by $\cO$. The number of p.p.a.v. with CM by $\cO$ is thus $\#\mathfrak{C}(\cO)$. 
Let us now compute the cardinality of $\#\mathfrak{C}(\cO)$. 
For this, we use the following sequence: 
\begin{eqnarray*}
0 \rightarrow (\cO_{K_0}^{\times})^+/N_{K/K_0}(\cO^{\times})\rightarrow \mathfrak{C}(\cO)\rightarrow \Cl(\cO)\xrightarrow{N_{K/K_0}} {\Cl^{+}(K_0)} \rightarrow 0.
\end{eqnarray*}
For $\cO=\cO_K$, the exactness of this sequence is proven in~\cite{Streng10} or~\cite{BroGru}. When $\cO\neq \cO_K$, the proof follows closely the lines of the proof for the maximal order for the exacteness at $(\cO_{K_0}^{\times})^+/N_{K/K_0}(\cO^{\times})$, $\mathfrak{C}(\cO)$ and $\Cl(\cO)$. For the surjectiveness of the norm map, $N_{K/K_0} :\Cl(\cO)\rightarrow \Cl^+(\cO_{K_0})$ we use the fact that it writes as a composition of two surjective maps $\Cl(\cO)\rightarrow \Cl(\cO_K)\rightarrow \Cl^+(\cO_{K_0})$.
\qed
\end{proof}
 
In the remainder of this paper, unless stated otherwise, we consider principally polarized abelian surfaces with complex multiplication by an order $\cO$ which has locally maximal real multiplication at $\fkl$, i.e. $\cO_{K_0,\fkl}\subset \cO_{\fkl}$. In this case, we may extend the notion of $\fkl$-isogeny. Indeed, if $A$ has CM by such an order, then the isogenies with kernel a subgroup of order $\ell$ of $A[\fkl]$ are called $\fkl$-isogenies by extension. 

\begin{lemma}
Let $A$ be an principally polarized abelian variety with locally maximal real multiplication at $\fkl$. Let $I:A\rightarrow B$ be a an $\fkl$-isogeny. Then $B$ is principally polarized and has locally maximal real multiplication at $\fkl$.
\end{lemma}
\begin{proof}
Note first that $\ell\cO_B\subset \cO_A$. Hence, if $f$ is the conductor of the real multiplication order of $A$ and $f'$ is the  real multiplication order of $B$, we have that $f \mid \ell f'$. Since $f$ is prime to $\ell$, it follows that $f\mid f'$.
Then following~\cite[§7.1, Prop. 7]{Shimura}, there are $f\cO_{K}$-transforms $\lambda_{A,f}:A\rightarrow A_{f}$ and $\lambda_{B,f}:B\rightarrow B_{f}$. We know that $A_{f}$ is principally polarized and has RM by $\cO_{K_0}$. Then there is a $\fkl$-isogeny $I'_{\fkl}:A_f\rightarrow B_{f}$ such that the  diagram in Figure~\ref{transforms} is commutative.

\begin{figure}[h!]\caption{}\label{transforms}
\begin{center}
\begin{tikzpicture}
    % \node (A) at (3:33){$A_{\fkl}$}; 
    % \node (B) at (2.75:37)  {$B_{\fkl}$}; 
    % \node (Ad) at (0:33) {$A$};
    % \node (Bd) at (0:37) {$B$};
    \node (A) at (0,0){$A$}; 
    \node (B) at (3,0)  {$B$}; 
    \node (Ad) at (0,2) {${A_f}$};
    \node (Bd) at (3,2) {${B_f}$};
    \draw[->] (A) to node[above] {\small{$I$}} (B);
    \draw[<-] (Ad) to node[left] {\small{$\lambda_{A,f}$}} (A);
    \draw[<-] (Bd) to node[right] {\small{$\lambda_{B,f}$}} (B);
    \draw[->] (Ad) to node[above] {\small{$I'$}} (Bd);
\end{tikzpicture}
\end{center}
\end{figure}

Since $I'$ is an $\fkl$-isogeny starting from an p.p.a.v. with RM by $\cO_{K_0}$, then it follows that $B_{\fkl}$ has RM by $\cO_{K_0}$. We conclude that $f\cO_{K_0}\subset \End(B)$, hence $B$ has locally maximal real multiplication at $\fkl$. \qed
\end{proof}

Let $I:A\rightarrow B$ be a separable isogeny between two p.p.a.v. defined over a perfect field. Denote by $\cO_A=\End(A)$ and $\cO_B=\End(B)$ and assume that these orders contain a suborder of $\cO_{K_0}$ which is locally maximal at $\fkl$. If $\cO_A\simeq \cO_B$, we say that the isogeny is \textit{horizontal}. If not, then the localizations of the orders at $\fkl$ lie on consecutive levels of the lattice given by Figure~\ref{fig:order-lattice}. If $(\cO_B)_{\fkl}$ is properly contained in $(\cO_A)_{\fkl}$, we say that the isogeny is \textit{descending}. In the opposite situation, we say the isogeny is \textit{ascending}.

Associated to an ordinary principally polarized abelian surface defined over a perfect field and whose endomorphism ring is an order with locally maximal real multiplication at $\fkl$, we define the $\fkl$-\textit{isogeny graph} to be the graph whose vertices are isomorphism classes of principally polarized abelian surfaces $(A, \iota, \lambda)$  with locally maximal real multiplication at $\fkl$ (following the notation from Section~\ref{GraphDefinition}) and whose edges are equivalence classes of $\fkl$-isogenies between these surfaces. Note that by Lemma~\ref{dualisogeny}, we may identify an abelian variety to its dual and consider this as non-oriented graph. With this terminology, we state our main results regarding the structure of the $\fkl$-isogeny graph for p.p.a.v defined over a number field. 

\begin{proposition}\label{splitting}
  Let $A$ be a principally polarizable abelian surface defined over $\C$, with endomorphism ring an order $\cO$ in a CM quartic field $K$ different from $\Q(\zeta_5)$. Let $\fkl$ be an ideal of prime norm $\ell$ in $\cO_{K_0}$ and assume that $\cO$ is locally maximal at $\fkl$.
    \begin{enumerate}
        \item Assume that $\fkl\cO_K$ is prime with the conductor of
            $\cO$, that we denote by $\fkf$. Then we have:
            \begin{enumerate}
                \item If $\fkl$ splits into two ideals in $\cO_K$, then
                    there are, up to an isomorphism,
                    exactly two horizontal \fkl-isogenies starting from
                    $A$ and all the others are descending.
                \item If $\fkl$ ramifies in $\cO_K$, up
                    to an isomorphism, there is exactly one horizontal
                    \fkl-isogeny starting from $A$ and all the others are
                    descending. 
                \item If $\fkl$ is inert in $K$, all $\ell+1$
                    $\fkl$-isogenies starting from $A$ are descending.
            \end{enumerate}
        \item If $\fkl$ is not coprime to $\fkf$, then up to an
            isomorphism, there is exactly one ascending $\fkl$-isogeny
            and $\ell$ descending ones starting from $A$.  
\end{enumerate}
\end{proposition}

\begin{proof}
(1) We treat first the case where $\cO$ is an $\cO_{K_0}$-order. If $\cO=\cO_K$, then the number of horizontal \fkl-isogenies between p.p.a.v. equals the number of ideal classes in the Shimura class group of $\cO_K$~given by ideals of norm $\ell$ (see~\cite[§7.5, Prop. 23]{Shimura} and~\cite[§14.4, Prop. 7]{Shimura}). %If $A$ and $B$ are defined over a finite field and have CM by the ring of integers $\cO_K$, this was proven by Waterhouse~\cite{Waterhouse}.
  Assume now that $\cO$ is an order of conductor $\mathfrak{f}$ prime to $\fkl$ and that there is an horizontal \fkl-isogeny between abelian varieties defined over $\C$ having endomorphism ring $\cO$. Following~\cite[§7.1, Prop. 7]{Shimura}, there are $\mathfrak{f}$-transforms towards $\lambda_{A,\fkl}:A\rightarrow A_{\fkl}$ and $\lambda_{B,\fkl} : B\rightarrow B_{\fkl}$, where $A_{\fkl}$ and $B_{\fkl}$
  have CM by $\cO_K$. Then there is an isogeny $I':A\rightarrow B$ such that the diagram in Figure~\ref{transforms} is commutative:
  
\begin{figure}\caption{}
\begin{center}
\begin{tikzpicture}
    % \node (A) at (3:33){$A_{\fkl}$}; 
    % \node (B) at (2.75:37)  {$B_{\fkl}$}; 
    % \node (Ad) at (0:33) {$A$};
    % \node (Bd) at (0:37) {$B$};
    \node (A) at (0,0){$A$}; 
    \node (B) at (3,0)  {$B$}; 
    \node (Ad) at (0,2) {${A_\fkl}$};
    \node (Bd) at (3,2) {${B_\fkl}$};
    \draw[->] (A) to node[above] {\small{$I$}} (B);
    \draw[<-] (Ad) to node[left] {\small{$\lambda_{A,\fkl}$}} (A);
    \draw[<-] (Bd) to node[right] {\small{$\lambda_{B,\fkl}$}} (B);
    \draw[->] (Ad) to node[above] {\small{$I'$}} (Bd);
\end{tikzpicture}
\end{center}
\end{figure}
Then $I'$ is an isogeny corresponding to a projective ideal $\fkl_1$,
lying over $\fkl$ in $\cO_K$. Since $\fkl$ is prime to the
conductor $\mathfrak{f}$, it follows that $I$ corresponds to the
ideal $\fkl_1\cap \cO$ in $\cO$. We conclude that the number of
horizontal $\fkl$-isogenies is $2$ if $\fkl$ is split in $K$, 1 if $\fkl$ is ramified in $K$ and 0 if $\fkl$ is inert.

In order to count descending isogenies, we count the abelian
    surfaces lying at a given level in the graph (up to isomorphism). To
    do this, let $\cO$ be the order of conductor $\fkf$ and assume $\fkf$ is prime to $\fkl$. We use Lemma~\ref{ShimuraCardinality} to compute $\#\mathfrak{C}(\cO)$, and thus the number of p.p.a.v. with CM by $\cO$, in terms of $\#\Cl(\cO)$. 

    To compute $\#\Cl(\cO)$, we will apply class number relations. More precisely, we have the exact sequence:
   
\begin{eqnarray}\label{exactSequence}
 1\rightarrow \cO^{\times}\rightarrow \cO_K^{\times}\rightarrow (\cO_K/\fkf\cO_K)^{\times}/(\cO/\fkf\cO_K)^{\times}\rightarrow \Cl(\cO)\rightarrow \Cl(\cO_K)\rightarrow 1. 
\end{eqnarray}
Hence we have the formula for the class number 
\begin{eqnarray*}
        \#\Cl(\cO)=\frac{\#\Cl(\cO_K)}{[\cO_K^{\times}:\cO^{\times}]}
        \frac{\#(\cO_K/\fkf\cO_K)^{\times}}{\#(\cO/\fkf\cO_K)^{\times}}.
\end{eqnarray*}
We have that $\cO_K^{\times}=\mu_K\cO_{K_0}^{\times}$, where $\mu_K=\{\pm 1\}$ (see~\cite[Lemma II.3.3]{Streng10}). Since $\cO_{K_0}\subset \cO$, it follows that
$[\cO_K^{\times}:\cO^{\times}]=1$.   
%\todo{Un petit mot sur $[\cO_K^{\times}:\cO^{\times}]$ ? C'est égal à
%$1$ car $\cO_{K_0}\subset\cO$, non ? S: Oui, j'ai mis la reference
%mais j'ai l'impression que pour $Q(\zeta_5)$ il se passe un truc
%pareil que pour les courbes elliptiques avec $j=0,1728$- il y plus
%d'automorphismes. Je ne comprends pas super bien ce que cela
%change.}

We note that $\cO/\fkf\cO_K\simeq \cO_{K_0}/\fkf\cO_{K_0}$. We denote by $N$ the norm of ideals in $\cO_{K}$. Moreover, we have that
    \begin{eqnarray}
        \#(\cO_K/\fkf{O}_K)^{\times}=
        N(\fkf)\prod_{\fkp\mathbin|\fkf }(1-\frac{1}{N(\fkp)}),
    \end{eqnarray}
    where the ideals in the product are all prime ideals of $\cO_K$,
    dividing the conductor. Let $\cO_{\fkl}$ be the $\cO_{K_0}$-order of
    conductor $\fkl \fkf$. By writing the exact
    sequence~\eqref{exactSequence} for the order $\cO_{\fkl}$, we obtain
    that 
\begin{align*}
        \#\Cl(\cO_{\fkl})&=\#\Cl(\cO)\frac{\#(\cO/\fkf\cO_K)^{\times}}
            {\#(\cO_{\fkl}/\fkf\fkl\cO_{K})^{\times}}\cdot \frac{\#(\cO_K/\fkf\fkl\cO_K)^{\times}}{\#(\cO_{K}/\fkf\cO_{K})^{\times}},\\
        &=\#\Cl(\cO)\frac1{\ell-1}N(\fkl)
            {\prod _{\fkp\mathbin|\fkl}(1-\frac{1}{N(\fkp)})},
            \end{align*}  
where we used the fact that $\#(\cO_l/\fkf\fkl\cO_K)^{\times}=(\ell-1)\times \#(\cO/\fkf\cO_K)^{\times}$.

Hence there are $(\ell-1)\cdot \#Cl(\cO)$ p.p.a.v. with CM by $\cO_{\fkl}$ when $\fkl$ is split in $K$, $\ell\cdot \#Cl(\cO)$ when $\fkl$ is ramified and $(\ell+1)\cdot \#Cl(O)$ when $\fkl$ is inert. 

 Moreover, by a simple symmetry argument, the number of descending isogenies
 is the same for every node lying at the $\cO$-level. Indeed,
let $A$ and $B$ be two nodes at the $\cO$-level. Then there is a
projective ideal $\fka$ in $\cO_K$ (which may be taken to be prime with
both $\fkl$ and $\fkf$), giving an horizontal isogeny $\lambda_{\fka}:A\rightarrow B$, corresponding to the ideal $\fka\cap \cO$. Assume that $A$ has a descending $\fkl$-isogeny towards a variety $A_{\fkl}$ lying at the $\cO_{\fkl}$-level $I_{\fkl}:A\rightarrow A_{\fkl}$. Note that there is a variety $B_{\fkl}$ at the $\cO_{\fkl}$-level such that $\lambda'_{\fka}:A_{\fkl}\rightarrow B_{\fkl}$ is the horizontal isogeny corresponding to the ideal $\fka\cap \cO_{\fkl}$. Then there is a $\fkl$-isogeny $I'_{\fkl}:B\rightarrow B_{\fkl}$ such that the following diagram is commutative: 
\begin{center}
\begin{tikzpicture}
    \node (A) at (0,0){$A_\fkl$}; 
    \node (B) at (3,0)  {$B_\fkl$}; 
    \node (Ad) at (0,2) {${A}$};
    \node (Bd) at (3,2) {${B}$};
    \draw[->] (A) to node[above] {\small{$\lambda'_{\fka}$}} (B);
    \draw[->] (Ad) to node[left] {\small{$I_{\fkl}$}} (A);
    \draw[dashed,->] (Bd) to node[right] {\small{$I'_{\fkl}$}} (B);
    \draw[->] (Ad) to node[above] {\small{$\lambda_{\fka}$}} (Bd);
\end{tikzpicture}
\end{center}

By comparing the number of abelian varieties at one level and the one below it and taking into account this symmetry, we conclude that all descending isogenies starting from each a.v. with CM by an order of conductor prime to $\fkl$ reach non-isomorphic nodes. Moreover, each a.v. with CM by the order of conductor $\fkl$ has exactly one ascending isogeny.

  Finally, let $\cO_0$ is an order in $K_0$ which is locally maximal at $\fkl$ and let $\fkf$ be the conductor of this order. Assume that $I:A\rightarrow B$ is an $\fkl$-isogeny and that $A$ and $B$ have real multiplication by $\cO_0$. Then there are $\fkf$-transforms from $A$ and $B$ towards two abelian varieties with maximal real multiplication. Then one obtains an $\fkl$-isogeny similar to the one in Figure~\ref{transforms} and $I$ has the same direction (i.e. horizontal, ascending or descending) as $I'$.

(2) If $\fkl$ divides $\fkf$, we have
    \begin{eqnarray*}
        \#\Cl(\cO_{\fkl})=\#\Cl(\cO)\frac{\#(\cO/\fkf\cO_K)^{\times}}
        {\#(\cO_{\fkl}/\fkf\fkl\cO_{K})^{\times}}.
    \end{eqnarray*}
    By a similar argument to the one above, we conclude that the number of ascending isogenies is 1 and the number of descending isogenies is $\ell$, for all p.p.a.v. with CM by the $\cO_{K_0}$-order of conductor $\fkf$.
\qed    
\end{proof}

\begin{remark}\label{ramifiedCase}
Note that Proposition~\ref{splitting} concerns also the case where $\ell$ is ramified in $\cO_{K_0}$. The structure of the $\fkl$-graph in this case is similar to the one in the split case.
\end{remark}

\begin{remark}
We have excluded the case $K=\Q(\zeta_5)$ because in this case $\cO_K^{\times}=\mu_K\cO_{K_0}$, where $\mu_K$ is the group of roots of unity with order 10. However, a nearly equivalent statement may be given for the structure of the $\fkl$-graph in this case. Only the degrees of vertices having CM by $\cO_K$ are affected. 
\end{remark}  
%Note that if there exists $\alpha\in K_0^+$ such that $\mathfrak{l}=\alpha O_{K_0}$, then this is an isogeny graph of principally polarizable abelian varieties. 
\noindent
Let $\pi$ be a $q$-Weil number in $K$, with $q$ prime to $\ell$. Suppose that in order to obtain a finite graph, we restrict to considering the subgraph of the $\fkl$-graph whose vertices are p.p.a.v. $A$ with endomorphism ring such that $\Z[\pi,\bar{\pi}]\subset \End(A)$ and whose edges are $\fkl$-isogenies between these abelian varieties. By extension, in the remainder of this paper, we will call this subgraph the $\fkl$-isogeny graph. From Proposition~\ref{splitting}, we deduce that the structure of a connected component of the $\fkl$-isogeny graph is exactly the one of an $\ell$-isogeny graph between elliptic curves, called \emph{volcano}~\cite{Kohel,FouMor}. By extension, in the remainder of this paper, we call this graph the $\fkl$-isogeny graph. Furthermore, we show in the following Section that this graph is isomorphic to a graph whose edges are ordinary p.p.a.v. defined over $\F_q$ and with maximal real multiplication, and whose edges are equivalence classes of rational $\fkl$-isogenies.   

\section{The structure of the real multiplication isogeny graph over
finite fields}\label{GraphStructure}
In this Section, we study the structure of the graph given by rational
isogenies between principally polarizable abelian surfaces defined over a
finite field, such that the corresponding endomorphism rings have locally maximal real multiplication at $\fkl$. The endomorphism ring of a p.p.a.v. $A$
over a finite field $\F_q$ ($q=p^n$) is an order in the quartic CM field
$K$ such that
\begin{eqnarray*}
\Z[\pi,\bar{\pi}]\subset \End(A) \subset \cO_K,
\end{eqnarray*}
where $\Z[\pi,\bar{\pi}]$ denotes the order generated by $\pi$, the
Frobenius endomorphism and by $\bar{\pi}$, the Verschiebung. Moreover,
the assumption that $\End(A)$ is an order with locally maximal real multiplication at $\fkl$ implies that its localization contains $(\cO_{K_0}\Z[\pi,\bar{\pi}])_{\fkl}$.

%\begin{lemma}
%    Let $J$ be an ordinary Jacobian over $\F_q$, isogenous to an abelian
%    variety $J'$ such that $\Z[\pi,\bar{\pi}] \subset \End(J')$. Then the
%    variety $J'$ is also defined over $\F_q$.
%\end{lemma}

\subsection{The $\fkl$-isogeny graph}
The notion of $\fkl$-isogeny defined in Definition~\ref{def:lfrak-isog} has
been used so far for abelian surfaces defined over $\bC$. We remark than
whenever an abelian variety defined over a finite field $\F_q$ has
endomorphism ring some order with locally maximal real multiplication at $\fkl$, we may define the notion of $\fkl$-isogeny exactly in the same way.

\begin{proposition}\label{reduction}
Let $\fkl$ be a prime ideal of degree 1 over $\ell$ in $\cO_{K_0}$, with $\ell\not =p$. Let $A$ be a principally polarized ordinary abelian variety defined over a finite field $\F_q$ of characteristic $p$ and having locally maximal real multiplication at $\fkl$. Then a $\fkl$-isogeny $I:A\rightarrow A'$ preserves real multiplication and the target variety $A'$ is principally polarizable. 
\end{proposition}

\begin{proof}
We choose a canonical lift $\tilde{A}$ of $A$ as defined in~\cite{Lubin}. We may assume without loss of generality that this is defined over a number field $L$~\cite{Chai}, such that $A$ is isomorphic to the reduction of $\tilde{A}$ modulo a ideal $\fkP$ lying over $p$ in $L$. We have that $A[\fkl]\simeq \tilde{A}[\fkl]$ and the reductions of $\fkl$-isogenies starting from $\tilde{A}$ give $\ell+1$ (equivalence classes of) isogenies starting from $A$. Hence there is an isogeny $\tilde{I}:\tilde{A}\rightarrow A_1$ such that $A_1$ has good reduction~$\pmod {\fkP}$ and its reduction is isomorphic to $A'$. Since the reduction map $\textrm{End}(A_1)\rightarrow \textrm{End}(A')$ is injective, it follows that $\textrm{End}(A')$ is an order with locally maximal real multiplication at $\fkl$. By reducing polarizations given in Proposition~\ref{alphapolarization} (see~\cite{Deligne} for the reduction of polarizations), we deduce that if $A$ is principally polarized, then $A'$ is also principally polarized.
\qed
\end{proof}

The following result is a generalization of \cite[Theorem 5 (ii) Chapter 13 §2]{LangEllipticFunctions}. The proof follows closely the lines of~\cite{LangEllipticFunctions}, but we reproduce it here for completeness.

\begin{lemma}\label{ConductorPrimes}
Let $A$ be an ordinary abelian variety defined over a finite field $\F_q$ of characteristic $p$. Then the prime $p$ does not divide the conductor of the order $\cO=\End(A)$.
\end{lemma}
\begin{proof}
Denote by $\cO_0$ the real multiplication order of $A$. Assume that $p$ divides the conductor $\fkf$ of $\cO$. Let $\pi$ be the Frobenius endomorphism of $A$. There is an element $\alpha \in \cO_{0}$ such that
\begin{eqnarray*}
\pi=\alpha+c\eta,
\end{eqnarray*} 
with $c\in \fkf$. Then we have
\begin{eqnarray*}
q\delta=\pi\bar{\pi}=\alpha^2\pmod \fkf,
\end{eqnarray*}
for some $\delta \in \cO^{\times}$. This implies that $p$ divides $\alpha$ in $\cO_{0}$. Since $A$ is ordinary and $A[p]\neq 0$, it follows that $\pi$ kills points of order $p$ in $A$. This is a contradiction, because $\pi$ is purely inseparable.  
\qed
\end{proof}

\noindent
The following result is a generalization of~\cite[Theorem 12 (b) Chapter 13 §4]{LangEllipticFunctions}. 

\begin{theorem}\label{CanonicalReduction}
Let $A$ be an abelian variety defined over a number field $L$, with complex multiplication by an order $\cO$. Let $\fkp$ be a prime ideal in $L$ over a prime number $p$, and assume that $A$ has good reduction $A=\bar{A} \pmod \fkp$ and that $\End_L(A)=\End(A)$. Let $\fkf$ be the conductor of $\cO$. Then if $p$ does not divide $\fkf$, the reduction map $\alpha\rightarrow \bar{\alpha}$ is an isomorphism of $\End(A)$ onto $\End(\bar{A})$. 
\end{theorem}

\begin{proof}
Let $\ell$ be a prime number. Let $S_{\ell}$ be the multiplicative monoid of positive integers prime to $\ell$ and let $\cO_{(\ell)}=S_{\ell}^{-1}\cO$ be the localization of $\cO$ at $\ell$.
First, we know from general theory on the reduction of abelian varieties~\cite{Shimura} that the reduction map 
\begin{eqnarray*}
\End(A)\rightarrow \End(\bar{A})
\end{eqnarray*}
is an injection. 
Since we have $T_{\ell}(A)\simeq T_{\ell}(\bar{A})$, for all $\ell \neq p$, it follows that $\End(A)$ and $\End(\bar{A})$ have the same localizations at $\ell$, by~\cite[Ch.13 \S 3 Lemma 1]{LangEllipticFunctions} (whose generalization to abelian varieties is straightforward). On the other hand, because $p$ does not divide the conductor $\fkf$, we have that $\cO_{K,(p)}\simeq \cO_{(p)}$, which means that $\cO_{(p)}$ is integrally closed. It follows that it will coincide with the localization at $p$ of $\End(\bar{A})$. We conclude that $\End(A)\simeq \End(\bar{A})$ because they have the same localization at all primes. 
\qed
\end{proof}

\vspace{0.5 cm}

With this in hand, we show that there is a graph isomorphism between an $\fkl$-isogeny graph between p.p.a.v. defined over finite fields and a certain graph whose vertices are p.p.a.v defined over a number field.  

\begin{corollary}\label{GraphIsomorphism}
Let $\mathcal G$ be an $\fkl$-isogeny graph with vertices principally polarized abelian surfaces defined over $\F_q$ and whose endomorphism ring is an order in $K$, with locally maximal real multiplication at $\fkl$. Let $\pi$ be a $q$-Weil number, giving the Frobenius endomorphism for any of the abelian surfaces in $\mathcal G$. Then $\mathcal{G}$ is isomorphic to a graph $\mathcal G'$, whose vertices are isomorphism classes of principally polarized abelian surfaces defined over a number field $L$ and having CM by an order whose localization at $\fkl$ contains $(\cO_{K_0}[\pi,\bar{\pi}])_{\fkl}$, and whose edges are equivalence classes of $\fkl$-isogenies between these surfaces.    
\end{corollary}

\begin{proof}
We choose $A_1$ and $A_2$ two abelian surfaces corresponding to vertices in the graph $\mathcal{G}$ connected by an edge $I:A_1\rightarrow A_2$ of kernel $G$ a cyclic group in $A_1[\fkl]$. We take $\tilde{A}_1$ the canonical lift of $A_1$. As explained in the proof of Lemma~\ref{reduction}, we may assume that there is a number field $L$ and a prime $\fkp$ lying over $p$ in $L$ such that $A_1$ is isomorphic to $\tilde{A}_1 \pmod \fkp$ and that $\End(\tilde{A}_1)$ is defined over $L$. We denote by $\mathcal G'$ the $\fkl$-isogeny graph whose vertices are p.p.a.v defined over $L$ whose endomorphism ring localized at $\fkl$ contains $(\cO_{K_0}[\pi,\bar{\pi}])_{\fkl}$. We will show that the graph $\mathcal{G}$ is isomorphic to $\mathcal{G}'$. There are $\ell+1$ cyclic groups in $\tilde{A}_1[\fkl]$ and we denote by $\tilde{G}$ the one such that the reduction of points gives an isomorphism $\tilde{G}\simeq G$. We consider the $\fkl$-isogeny of kernel $\tilde{G}$, $I':\tilde{A}_1\rightarrow A_2'$. Then $A_2'\pmod \fkp=A_2$ and we know that $\End(A_2')\rightarrow \End(A_2)$ is an injection. Since $p$ does not divide the conductor of $\End(\tilde{A}_1)$ (because by Lemma~\ref{ConductorPrimes} it does not divide the conductor of $A_1$), it follows $p$ cannot divide the conductor of $A'_2$. Hence by Theorem~\ref{CanonicalReduction} it follows that $A_2'$ is isomorphic to the canonical lift of $A_2$.  
\qed
\end{proof}

\bigskip

We are now interested in determining the field of definition of
$\fkl$-isogenies starting from a p.p.a.v $A$. For that, we need several definitions. 

\phantomsection
Let $\fkl$ be an ideal in $\cO_{K_0}$ and $\alpha$ a generator of this
ideal. Let $\cO$ be an order of $K$ with locally maximal real multiplication at $\fkl$ and let $\theta\in \cO$. We define
the $\fkl$-adic exponent of $\theta$ in $\cO$ with respect to $\cO_{K_0}$ as
$$\label{fkl-adic-valuation}\nu_{\fkl,\cO}(\theta)\coloneq\max _{m\geq
0}\{m:\theta_{\fkl} \in \cO_{K_0,\fkl}+\fkl^m\cO_{\fkl}\},$$ 
where $\theta_{\fkl}$ is the image of $\theta$ via the homomorphism $\cO\rightarrow \cO_{\fkl}$. 
Recall that for a p.p.a.v $A$ with maximal real multiplication, we are 
interested (by Lemma~\ref{lemma:ordlat-2dim}) in computing the $\fkl$-adic valuation of the conductor of the endomorphism ring $\cO_A$. We remark that it suffices to determine $\nu_{\fkl,\cO_A}(\pi)$. Indeed, we have
$\cO_{A,\fkl}=\cO_{K_0,\fkl}+\fkf_{\eta,\cO_A}\eta_{\fkl}$ and 
\begin{eqnarray}\label{conductorValuation}
\nu_{\fkl}(\fkf_{\eta,\cO_A})=
    \nu_{\fkl,\cO_K}(\pi)-\nu_{\fkl,\cO_A}(\pi).
\end{eqnarray}
In the remainder of this paper, we denote by $\nu_{\fkl,A}(\pi)\coloneq
\nu_{\fkl,\cO_A}(\pi)$.

\begin{example}\label{ex:smallGraph}
    Let $H$ be the genus-2 curve given by the equation
    $$y^2=31x^6+79x^5+109x^4+130x^3+62x^2+164x+56$$ defined over
    $\F_{211}$. The Jacobian $J$ has complex multiplication by a quartic
    CM field $K$ with defining equation $X^4+81X^2+1181$. The real
    subfield is $K_0=\Q(\sqrt{1837})$, and has class number~1.  The
    endomorphism ring of $J$ contains the real maximal order $\cO_{K_0}$.
    In the real subfield $K_0$, we have $3=\alpha_1\alpha_2$, with
    $\alpha_1=\frac{43+\sqrt{1837}}{2}$ and $\alpha_2$ its conjugate.
    %The 3-torsion is defined over an extension field of degree 6, but
    %$J[\alpha_1]\subset J(\F_{q^6})$ and $J[\alpha_2]\subset J(\F_{q^2})$. 
   We have that $\nu_{(\alpha_i),\cO_K}(\pi)=1$, for $i=1,2$,
    where $\pi$ has relative norm $211$ in $\cO_K$. 
\end{example}

Since $\fkl$ is principal in the real multiplication order of $A$, it follows that $A[\fkl]$ is the kernel of an endomorphism.
Since $A$ is ordinary, all endomorphisms are $\F_q$-rational.
Consequently, we have that $\pi(A[\fkl^n])\subset A[\fkl^n]$, for $n\geq 0$. The following result relates the computation of the $\fkl$-adic exponent of $\pi$ to that of the matrix of the Frobenius on the $\fkl$-torsion.  

\begin{proposition}\label{rationality}
 Let $A$ be a p.p.a.v. defined over a finite field $\F_q$ and having CM by an order with locally maximal real multiplication at $\fkl$. Then the largest integer $n$ such that the Frobenius matrix on $A[\fkl^n]$ is of the form
    \begin{eqnarray}\label{FrobeniusMatrix}
        \left (\begin{array}{cc}
            \lambda & 0\\
            0 & \lambda\\
        \end{array}
        \right )\bmod \ell^n
    \end{eqnarray}
    is $\nu_{\fkl,A}(\pi)$.
\end{proposition}
\begin{proof}
    First we assume that $\nu_{\fkl,A}(\pi)=n$ and we show
    that the matrix of the Frobenius has the form given by
    equation~\eqref{FrobeniusMatrix}. Let $D$ be an element of $A[\fkl^n]$. Then $\pi$ acts on $D$ as an element of $\cO_{K_0}/\fkl^n\simeq \Z/\ell^n\Z$.
    Hence $\pi(D)=\lambda D$ for some $\lambda \in \Z$. 
    
Conversely, suppose that the matrix of the Frobenius on $A[\fkl^n]$ is of the form~\eqref{FrobeniusMatrix} and take $\alpha$ a real multiplication endomorphism such that $\alpha(D)=\lambda D$, for all $D\in A[\fkl^n]$ (Since any real multiplication endomorphism acts on $A[\fkl^n]$ as $\lambda I_2$, it is easy to see that such an $\alpha$ exists). Then $\pi-\alpha$ is zero on $A[\fkl^n]$, which implies that this is an element of $\fkl^n\cO$ (by~\cite[Proposition 7]{EisLau}).

% The minimal polynomial of the
 %   Frobenius splits as $(T^2-\lambda_1T+p)(T^2-\lambda_2T+p)$, with
 %   $\lambda_1,\lambda_2\in K_0$ (see for instance~\cite[Theorem
 %   1.2]{Ruck}). It follows that the eigenvalue of the Frobenius
 %   restricted to $J[\fkl_i]$ is such that $\lambda^2=p \pmod {\fkl_i}$. As
 %   a consequence, the matrix for the Verschiebung is
 %   \begin{eqnarray}
 %       \left (\begin{array}{cccc}
 %           \lambda & 0 \\
 %           0 & \lambda \\
 %       \end{array}
 %       \right ).
 %   \end{eqnarray}
 %   Hence $\nu_{\fkl_i,J}(\pi)=n$.
%}
\qed
\end{proof}

\begin{remark}\label{firstRemark} Let $\F_{q^k}$ be the smallest field extension such that $A[\fkl^n]$ is defined over $\F_{q^k}$. A natural consequence of
Proposition~\ref{rationality} is that the cyclic subgroups of
$A[\fkl^n]$ are rational (i.e. stable under the action of $Gal(\F_{q^k}/\F_{q})$) if and only if $\nu_{\fkl,A}(\pi)\geq n$. In particular, the $\ell+1$ isogenies whose kernels are cyclic subgroups of $A[\fkl]$ are rational
if and only if $\nu_{\fkl,A}(\pi)>0$.
\end{remark}

By Proposition~\ref{splitting} and Corollary~\ref{GraphIsomorphism} we get the following structure of connected components of the non-oriented $\fkl$-isogeny graph over finite fields. 
   
\begin{enumerate}
    \item At each level, if $\nu_{\fkl,A}(\pi)>0$, there are
        $\ell+1$ rational isogenies with kernel a cyclic subgroup of
        $A[\fkl]$.
    \item If $\fkl$ is split in $\cO_{K_0}$ then there are two
        horizontal $\fkl$-isogenies at all levels such that the
        corresponding order is locally maximal at $\fkl$ (i.e. $\cO_{\fkl}\simeq \cO_{K,\fkl}$). At every intermediary level (i.e. $\nu_{\fkl,A}(\pi)>0$), there are $\ell+1$ rational $\fkl$-isogenies: an ascending one and $\ell$ descending ones.
    \item If $\nu_{\fkl,A}(\pi)=0$, then no smaller order
        (whose conductor has larger $\fkl$-valuation) contains
        $\pi$. There are no rational descending
        $\fkl$-isogeny, and there is exactly one ascending
        $\fkl$-isogeny.
\end{enumerate}

%In particular, Remark~\ref{firstRemark} implies that if an $\fkl$-isogeny $I:A_1\rightarrow A_2$ is such that $\cO_{K_0}[\pi,\bar{\pi}]\subset \End (A_1)$ and $\cO_{K_0}[\pi,\bar{\pi}]\subset \End (A_2)$, then $I$ is rational. 
We will show that all rational isogenies of degree $\ell$ preserving locally maximal real multiplication at $\fkl$ are $\fkl$-isogenies, for some ideal $\fkl$ of degree 1. 

\begin{lemma}\label{cyclicValuation}
    Let $A$ and $B$ be two abelian varieties defined and isogenous over
    $\F_q$ and denote by $\cO_A$ and $\cO_B$ the corresponding
    endomorphism rings. Let $\fkl$ be an ideal of norm $\ell$ in
    $\cO_{K_0}$. Assume that the $\fkl$-adic valuations of the conductors
    of $\cO_A$ and $\cO_B$ are different. Then for any isogeny
    $I:A\rightarrow B$ defined over $\F_q$ we have $\Ker I\cap
    A[\fkl]\neq \{0\}$.
\end{lemma}

\begin{proof}
We prove the contrapositive statement. Assume that there is an
isogeny $I:A\rightarrow B$ defined over $\F_q$ with $\Ker I\cap
A[\fkl]= \{0\}$.  We then have that $I(A[\fkl^n])= B[\fkl^n]$, for
all $n\geq 1$. Since $\pi_B\circ I=I\circ \pi_A$, it follows that the
$\fkl$-adic exponents $\nu_{\fkl,\cO_A}(\pi_A)$ and $\nu_{\fkl,\cO_B}(\pi_B)$ are equal. By equation~\eqref{conductorValuation}, it follows that the $\fkl$-adic valuations of the conductors of endomorphism rings of $A$ and $B$ are equal.\qed
\end{proof}

The converse of Lemma~\ref{cyclicValuation} does not hold, as it is
possible for an \fkl-isogeny to have a kernel within $A[\fkl]$, and yet
leave the \fkl-valuation of the conductor of the endomorphism ring
unchanged. The following statement is a converse to Proposition~\ref{prop:ell-isog-preserve-rm}.

\begin{proposition}\label{allRMIsogenies}
    Let $\ell$ be an odd prime number, split in $K_0$. All cyclic
    isogenies of degree~$\ell$ between principally polarizable abelian varieties defined over $\F_q$ having locally maximal real multiplication at $\ell$ are $\fkl$-isogenies, for some degree 1 ideal $\fkl$ in $\cO_{K_0}$.
\end{proposition}

\begin{proof}
    Let $\ell \cO_{K_0}=\fkl_1\fkl_2$. Let $I:A\rightarrow B$ be a
    rational degree-$\ell$ isogeny which preserves the real
    multiplication $\cO_{K_0}$. The endomorphism rings $\cO_A$ and
    $\cO_B$ are orders whose localizations are located in the lattice of orders described by Figure~\ref{fig:order-lattice}. First, by~\cite[Section 8]{BroGru},
    we have that either $\ell \cO_A\subset \cO_B$, and $\ell \cO_B\subset \cO_A$. Hence the two orders lie either on the same level, either on  consecutive levels in the lattice of orders. If $\cO_A$ and $\cO_B$ lie on consecutive levels, then there is an ideal $\fkl$ of norm  $\ell$ in $\cO_{K_0}$ such that the $\fkl$-adic valuation of the  conductors is different. By Lemma~\ref{cyclicValuation}, it follows that the kernel of any cyclic $\ell$-isogeny between $A$ and $B$ is a
 cyclic subgroup of $A[\fkl]$.

    Assume now that $\cO_A$ and $\cO_B$ are $\O_{K_0}$-orders and that they lie at the same level in the lattice of orders. Then by using the Shimura class group action, an horizontal isogeny between $A$ and $B$ corresponds to an invertible ideal $\fku$ of $\cO_A$. Moreover we have $\fku{\bar{\fku}}=\fkl$, with \fkl\ an ideal of norm $\ell$ in $\cO_{K_0}$. Hence it is an $\fkl$-isogeny, for some ideal $\fkl$. Secondly, if $\cO_A$ and $\cO_B$ contain a suborder of $\cO_{K_0}$ of conductor $f$ prime to $\ell$, then we consider $f\cO_K$-transforms towards abelian varieties with RM by $\cO_{K_0}$ and reduce the problem to the first case.

 Finally, if the two orders lie at the same level and are not isomorphic, then
 both the $\fkl_1$-adic and $\fkl_2$-adic valuations of the corresponding conductors are different. It then follows that the kernel of any isogeny from $A$ to $B$ contains a subgroup of $A[\fkl_1]$ and one of $A[\fkl_2]$. This is not possible if the isogeny is cyclic. \qed
\end{proof}

\subsection{The $\ell$-isogeny graph}

Associated to an ordinary principally polarizable abelian surface defined over $\F_q$ and having locally maximal real multiplication at $\ell$, we define the $\{\fkl_1,\fkl_2\}$-\textit{isogeny graph} to be the labeled graph whose edges are all equivalence classes of $\fkl_1$- or $\fkl_2$-isogenies, and whose vertices are isomorphism classes of principally polarizable abelian surfaces over $\F_q$ reached (transitively) by such isogenies. An edge is labeled as $\fkl_1$ or $\fkl_2$, if it corresponds to a $\fkl_1$-isogeny or to a $\fkl_2$-isogeny, respectively.

%\textcolor{red}{Je crois que nous n'avons pas montre que la duale d'une \fkl-isogenie est aussi une \fkl-isogenie}
A natural consequence of Proposition~\ref{allRMIsogenies} is that over
finite fields, the $\{\fkl_1,\fkl_2\}$-isogeny graph is the graph of all
isogenies of degree $\ell$ between principally polarizable abelian surfaces
having locally maximal real multiplication at $\ell$. 

Note that the $\{\fkl_1,\fkl_2\}$-isogeny graph can be seen, by the results above, as the union of two graphs which share all their characteristics with genus one isogeny volcanoes. In particular, the generalization of the top rim of the volcano turns into a torus if both $\fkl_1$ and $\fkl_2$ split. If only one of them splits, the top rim is a circle, and if both are inert we have a single vertex corresponding to a maximal endomorphism ring
(since all cyclic isogenies departing from that abelian variety increase
both the $\fkl_1$- and the $\fkl_2$-valuation of the conductor of the
endomorphism ring).

%Let $A$ be an abelian variety with CM by an order $\cO$ in $K$ and $\mathfrak{c}$ be an ideal of $\cO$. Then 
%\begin{eqnarray*}
%\C/\RPhi(\mathfrak{a})\rightarrow \C/\RPhi(\mathfrak{ca})
%\end{eqnarray*}
%gives a horizontal isogeny of norm $N_{K/\Q}(\mathfrak{c})$. 

%\begin{proposition}
%Let $I:A\rightarrow B$ be an isogeny of degree $\ell$ between ordinary abelian varieties with CM by a field $K$ and let $\mathcal{O}_A=End(A)$ and $\mathcal{O}_B=End(B)$. Then $\mathcal{O}_A\simeq \mathcal{O}_B\simeq \mathcal{O}$ if and only if the isogeny $I$ corresponds to an ideal in $\mathcal{O}$. 
%\end{proposition}

%\vspace{0.5 cm}

\subsubsection*{MAGMA experiments.}

Let $A$ be a p.p.a.v. defined over $\F_q$ with maximal real
multiplication at $\ell$. We do not have formulas for computing cyclic isogenies
over finite fields (Section~\ref{sec:algorithm} works around this
difficulty for the computation of endomorphism rings). Instead, we
experiment over the complex numbers, and use the graph isomorphism between the $\fkl$-isogeny graph having $A$ as a vertex and the graph of its canonical lift.

To draw the graph corresponding to Example~\ref{ex:smallGraph}, it is
straightforward to compute the period matrix $\Omega$ associated to a
complex analytic torus $\C^2/\Lambda_1+\tau\Lambda_2$, and compute a
representative in the fundamental domain for the action of $\Sp_4$ using
Gottschling's reduction algorithm\footnote{By Gottschling's reduction
algorithm, we refer to the reduction algorithm as stated
in e.g.~\cite[chap. 6]{Dupont06} or~\cite[\S6.3]{Streng14}, and which relies crucially on Gottschling's work~\cite{Gottschling59}
for defining the 19 matrices which come into play}.

All this can be done symbolically, as the matrix $\Omega$ is defined over
the reflex field $K^r$. As a consequence, we may compute isogenies of
type~\eqref{John} and follow the edges of the graph of isogenies between
complex abelian surfaces having complex multiplication by an order $\cO$
containing $\cO_{K_0}[\pi,\bar{\pi}]$. The exploration terminates when
outgoing edges from each node have been visited. This yields
Figure~\ref{fig:SmallGraph}. Violet and orange edges in
Figure~\ref{fig:SmallGraph}  are $\alpha_1$ and $\alpha_2$-isogenies,
respectively. Note that since $\alpha_1$ and $\alpha_2$ are totally
positive, all varieties in the graph are principally polarized.

%\parfillskip=0pt\relax\hfil\hbox{\emph{End of Example \arabic{example}.}}

%\todo{Voir s'il n'est pas pertinent de définir comme noeuds du graphes
%les tores, sans polarisation particulière...}

\begin{figure}
    \begin{center}
        % \raisebox{1.7cm}{\includegraphics[scale=0.6]{pictures/lattice.pdf}}\qquad
\includegraphics[scale=0.45,angle=90]{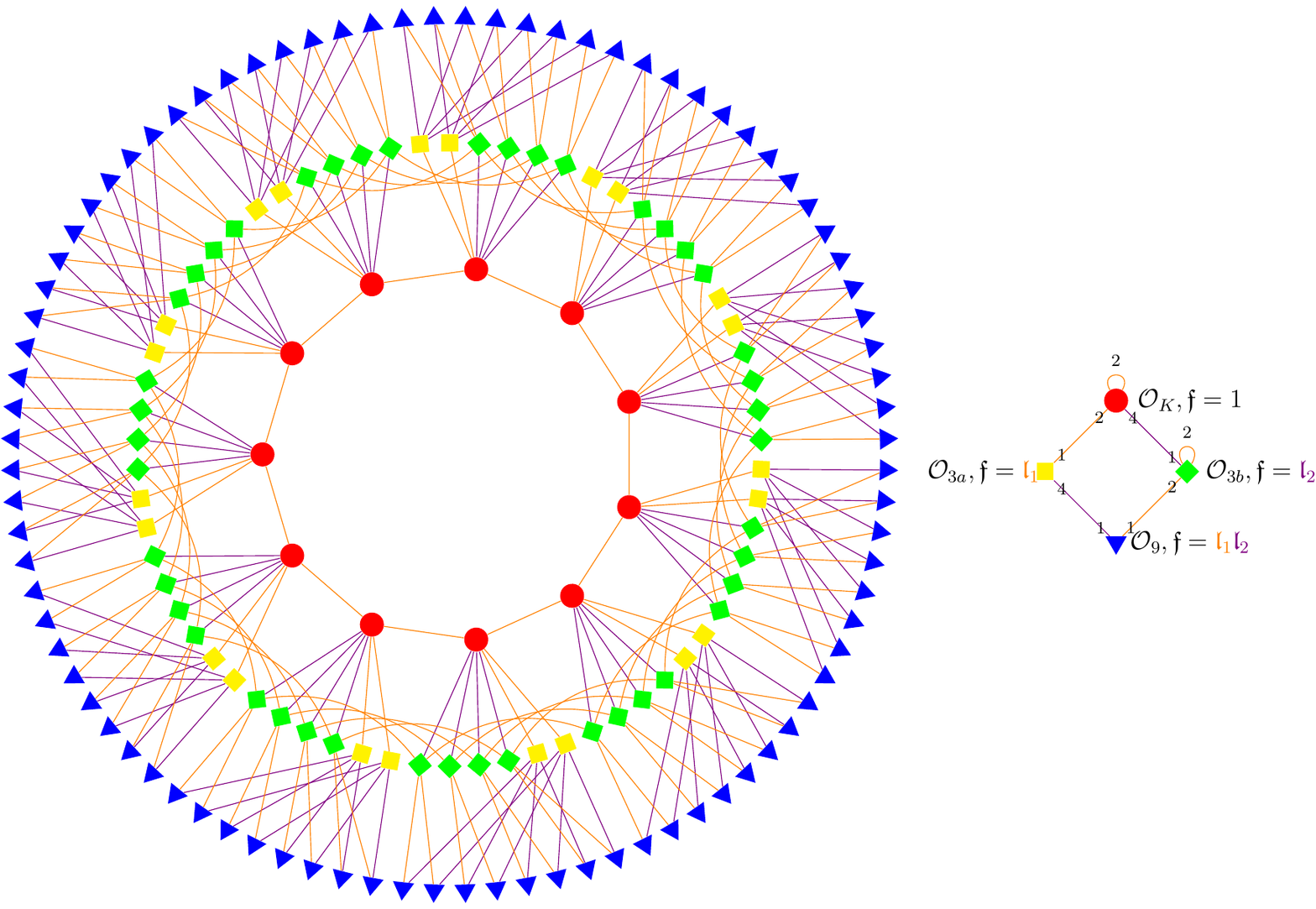}
\caption{\label{fig:SmallGraph}Graph of $\ell$-isogenies preserving real
        multiplication, for $\ell=3$, $K$ defined by
        $\alpha^4+81\alpha^2+1181$, and $\cO_{K_0}[\pi]$
        defined by the Weil number $\pi=\frac12(\alpha^2 + 3\alpha +
        45)$, with $p=\Norm_{K/K_0}\pi=211$.}
    \end{center}
\end{figure}

\vspace{0.5 cm}
In a computational perspective, we are interested in
$\ell$-isogenies, which are accessible to computation using the
algorithms developed by~\cite{CosRob}. Our description of the $\fkl_1$-
and $\fkl_2$-isogenies is key to understanding the
$\ell$-isogenies due to the following result.

\begin{proposition}\label{LLisogenies}
Let $\ell \geq 3$ be a prime number such that $\ell\cO_{K_0}=\fkl_1\fkl_2$. Then all $\ell$-isogenies between p.p.a.v. defined over $\F_q$ and having locally maximal real multiplication at $\ell$ are a composition of an $\fkl_1$-isogeny with an $\fkl_2$-isogeny.
\end{proposition}

\begin{proof}
    Let $A$ and $B$ be two p.p.a.v. defined over $\F_q$ and let $I:A\rightarrow B$ be an $\ell$-isogeny preserving the real multiplication order, which is locally maximal at $\ell$. We denote by $\cO_A=\End(A)$ and $\cO_B=\End(B)$. If the endomorphism rings are both isomorphic to an order $\cO_{K_0}$-order denoted by $\cO$, then the isogeny corresponds, under the action of the Shimura class group $\fkC(\cO)$, to an ideal class $\fka$ such that $\fka\bar{\fka}=\ell \cO$. It follows that both $\fkl_1$ and $\fkl_2$ split in $\cO$. Let $\fkl_{i,j}$, $i,j\in \{1,2\}$, be such that $\fkl_{i,1}\fkl_{i,2}=\fkl_i$. Then, we may assume that the isogeny $I$ corresponds to the ideal $\fkl_{1,1}\fkl_{2,1}$ under the action of the Shimura class group. We conclude that $I$ is a composition of an $\fkl_1$-isogeny with an $\fkl_2$-isogeny. \\
If $\cO_A$ and $\cO_B$ are isomorphic to an order in $K$ which contains a suborder of $\cO_{K_0}$ of conductor $f$ prime to $\ell$, then the result follows by choosing $f\cO_K$-transforms and reducing the problem to finding a horizontal isogeny between two p.p.a.v. with CM by $\cO_K$, as in the proof of Proposition~\ref{splitting}.\\
Assume now that $\cO_A$ and $\cO_B$ are not isomorphic. This implies that $\nu_{\fkl,\cO_A}(\pi)$ and $\nu_{\fkl,\cO_B}(\pi)$ differ for some $\fkl$, and we may without loss of generality assume $\fkl=\fkl_1$. By considering the dual isogeny $\hat I$ instead of $I$, we may also assume $\nu_{\fkl_1,\cO_A}(\pi)>\nu_{\fkl_1,\cO_B}(\pi)$.\\
Let $n=\nu_{\fkl_1,\cO_A}(\pi)$. We then have that any cyclic subgroup of $A[\fkl_1^n]$ is rational. By Proposition~\ref{rationality}, there is a cyclic subgroup of $B[\fkl_1^{n}]$ which is not rational. Since $I(A[\fkl_1^n])\subset B[\fkl_1^n]$ and the isogeny $I$ is rational, it follows that $\Ker I$ contains an element $D_1\in A[\fkl_1]$. Let $I_1:A\rightarrow C$ be the isogeny whose kernel is generated by $D_1$. This isogeny preserves the real multiplication and is an $\fkl_1$-isogeny (Proposition~\ref{allRMIsogenies}). By~\cite[Prop 7]{EisLau}, there is an isogeny $I_2:C\rightarrow B$ such that $I=I_2\circ I_1$. Obviously, $I_2$ also preserves real multiplication.\\
Let now $\langle D_1,D_2\rangle=\Ker I$. Since $\Ker I\subset A[\fkl_1]+A[\fkl_2]$, we may write $D_2=D_{2,1}+D_{2,2}$ with $D_{2,i}\in A[\fkl_i]$. As $\Ker I$ is Weil-isotropic, we may choose $D_2$ so that $D_{2,1}=0$, whence $D_2\in A[\fkl_2]$. We have $I_1(D_2)\not=0$, so that $I_2$ is an $\fkl_2$-isogeny.
 Note that given the $D_2\in A[\fkl_2]$ which we have just defined, we may also consider the $\fkl_2$-isogeny $I'_2:A\rightarrow C'$ with kernel $\langle D_2\rangle$, and similarly define the $\fkl_1$-isogeny $I'_1$ which is such that $I=I'_1\circ I'_2$.

\qed
\end{proof}

The proposition above leads us to consider properties of
$\ell$-isogenies with regard to the $\fkl_i$-isogenies they are
composed of. Let $I=I_1\circ I_2$ be an $\ell$-isogeny, with $I_i$
an $\fkl_i$-isogeny (for $i=1,2$). We say that $I$ is $\fkl_1$-ascending
(respectively $\fkl_1$-horizontal, $\fkl_1$-descending) if the
$\fkl_1$-isogeny $I_1$ is ascending (respectively horizontal,
descending). This is well-defined, since by Lemma~\ref{cyclicValuation}
there is no interaction of $I_1$ with the $\fkl_2$-valuation of the
conductor of the endomorphism ring.

Proposition~\ref{LLisogenies} is a way to interpret Figure~\ref{llgraph}
as derived from Figure~\ref{fig:SmallGraph} as follows. Vertices are
kept, and we use as edges all compositions of one $\fkl_1$-isogeny and
one $\fkl_2$-isogeny. This fact will serve as a basis for our algorithms
for computing endomorphism rings, detailed in
Section~\ref{sec:algorithm}.

\section{Pairings on the real multiplication isogeny
graph}\label{PairingTheGraph}

Let $\ell\cO_{K_0}=\fkl_1\fkl_2$. In this Section, $\fkl$ denotes any of the ideals $\fkl_1, \fkl_2$. Let $A$ be a p.p.a.v. defined over $\F_q$, with complex multiplication by an order which is locally maximal at $\fkl$.

We relate some properties of the Tate pairing to the isomorphism class of
the endomorphism ring of the abelian variety, by giving a similar result to the
one of~\cite{IonJou2} for genus-1 isogeny graphs. More precisely, we
show that the nondegeneracy of the Tate pairing restricted to the kernel
of an $\fkl$-isogeny determines the type of the isogeny in the graph, at
least when $\nu_{\fkl}(\pi)$ is below some bound. This result
is then exploited to efficiently navigate in isogeny graphs. 

Let $r$ be the smallest integer such that $A[\fkl]\subset A(\F_{q^r})$.
Let $n$ be the largest integer such that $A[\fkl^n]\subset A[\F_{q^r}]$.
We define $k_{\fkl,A}$ to be
\begin{eqnarray*}
    k_{\fkl,A}=\max_{P\in A[\fkl^n]}
        \{k\mid T_{\ell^n}(P,P) \in \mu_{\ell^k}\backslash \mu_{\ell^{k-1}}\}.
\end{eqnarray*}

\begin{definition}
    Let $G$ be a cyclic group of $A[\fkl^n]$. We say that the Tate
    pairing is $k_{\fkl,A}$-non-degenerate (or simply non-degenerate) on
    $G\times G$ if its restriction
    \begin{eqnarray*}
        T_{\ell^n}:G\times G\rightarrow \mu_{\ell^{k_{\fkl,A}}}
    \end{eqnarray*}
    is surjective. Otherwise, we say that the Tate pairing is
    $k_{\fkl,A}$-degenerate (or simply degenerate) on $G\times G$.
\end{definition}

 The following result shows that computing the $\fkl$-adic valuation
of $\pi$ is equivalent to computing $k_{\fkl,A}$. 

\begin{proposition}\label{Equality}
    Let $r$ be the smallest integer such that $A[\fkl]\subset
    A(\F_{q^r})$. Let $n$ be the largest integer such that
    $A[\fkl^n]\subset A[\F_{q^r}]$. Then if $\nu_{\fkl,A}(\pi^r)<2n$, we have
    \begin{eqnarray*}\label{equality}
        k_{\fkl,A}=2n-\nu_{\fkl,A}(\pi^r).
    \end{eqnarray*}
\end{proposition}
\begin{proof}
    Let $Q_1,Q_2$ form a basis for $A[\fkl^{2n}]$. Then $\pi^r(Q_i)=\sum
    a_{ij}Q_{j}$, for $i,j={1,2}$. We have
%\todo{Dans l'intro, on a $F_P(\pi)$ sur la première coordonnée.}
\begin{eqnarray*}
    T_{\ell^n}(\ell^nQ_i,\ell^nQ_i)=
        W_{\ell^{2n}}(\pi(Q_i)-Q_i, Q_i)=W_{\ell^{2n}}(Q_k,Q_{i})^{a_{ik}}\in \mu_{\ell^{k_{\fkl,A}}},
\end{eqnarray*}
    with $k\equiv i+1 \pmod 2$. By the non-degeneracy of the Weil
    pairing, this implies $a_{12}\equiv a_{21}\equiv 0\pmod
    {\ell^{2n-k_{\fkl,A}}}$. Moreover, the antisymmetry condition on the
    Weil pairing says that
\begin{eqnarray*}
T_{\ell^n}(\ell^nQ_1,\ell^nQ_2)T_{\ell^n}(\ell^nQ_2,\ell^nQ_1)\in
\mu_{\ell^{k_{\fkl,A}}}.
\end{eqnarray*}
    Since $T_{\ell^n}(\ell^n Q_i,\ell^n
    Q_j)=W_{\ell^{2n}}(Q_i,Q_j)^{a_{jj}-1}$, for $i\neq j$, we have that
    \begin{eqnarray*}
        W_{\ell^{2n}}(Q_1,Q_2)^{a_{11}-1}W_{\ell^{2n}}(Q_2,Q_1)^{a_{22}-1}=
        W_{\ell^{2n}}(Q_1,Q_2)^{a_{11}-a_{22}}\in \mu_{\ell^{k_{\fkl,A}}}.
    \end{eqnarray*}
    We conclude that $\ell^{2n-k_{\fkl,A}}$ divides all of $a_{12}$,
    $a_{21}$, and $a_{11}-a_{22}$. By Proposition~\ref{rationality},
    this implies that $2n-k_{\fkl,A}\leq\nu_{\fkl,A}(\pi^r)$.
    
    Conversely, let $k=2n-\nu_{\fkl,A}(\pi^r)$. We know (by
    Proposition~\ref{rationality}) that $\pi=\lambda I_2+\ell^{2n-k}A$,
    for $A\in M_2(\Z)$ and for some $\lambda$ coprime to $\ell$.  Then
    for $P\in A[\fkl^n]$ and $\bar{P}$ such that $\ell^n\bar{P}=P$, we
    have $T_{\ell^n}(P,P)=W_{\ell^{2n}}(\bar{P},\lambda \bar
    {P}+A(\ell^{2n-k} \bar{P}))\in \mu_{\ell^{k}}$. Hence $k\geq
    k_{\fkl,A}$ and this concludes the proof. \qed
\end{proof}

From this proposition, it follows that if
$\nu_{\fkl,A}(\pi)>2n$, the self-pairings of all kernels of
$\fkl$-isogenies are degenerate. At a certain level in the $\fkl$-isogeny graph,
when $\nu_{\fkl,A}(\pi)<2n$, there is at least one kernel with
non-degenerate pairing (i.e. $k_{\fkl,A}=1$). Following the terminology
of~\cite{IonJou}, we call this level \textit{the second stability level}.
As we descend to the floor, $k_{\fkl,A}$ increases. The \textit{first
stability level} is the level at which $k_{\fkl,A}$ equals $n$. 

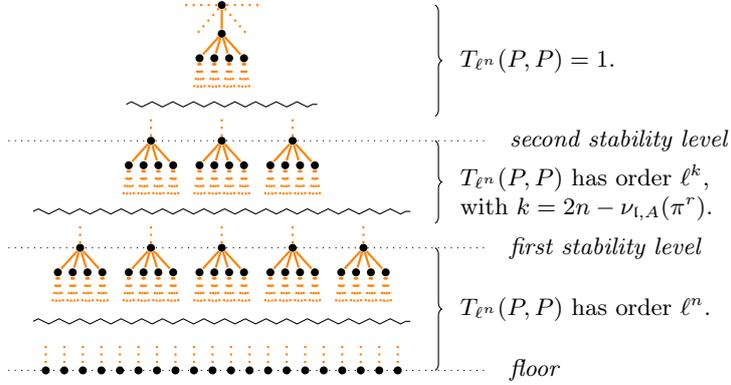
\begin{figure}
\begin{center}
    %\todo{Les légendes ne vont que moyennement. Le $n$ n'est pas le même
    %pour tous les $A$.}
    %\todo{Ae suis pas sûr de moi concernant les $<$ et les $\leq$.}
    \begin{tikzpicture}[
            point/.style={circle, black, fill, minimum size=3pt, inner sep=0pt},
            every node/.style={point},
            text node/.style={rectangle, draw=none, fill=none, minimum size=0pt, inner sep=0pt},
            mu1/.style={orange, thick},
            every child/.style={mu1},
            mu2/.style={violet, thick},
            cutline/.style={decorate,decoration={name=zigzag,amplitude=0.25ex,segment
            length=2ex}},
            level 1/.style={sibling distance=0.6em, level distance=1em},
            level 2/.style={dotted,sibling distance=0.15em, level distance=1.25em},
            % mu1/.style={thick},
            % mu2/.style={thick},
        ]

        \node (A) {}
                child { node {} child child child child }
                child { node {} child child child child }
                child { node {} child child child child }
                child { node {} child child child child }
        ;
        \node[above=0.8em of A] (R) {};
        \draw[mu1] (R)--(A);
        \draw[mu1,dotted]
            (R) -- ++(1.5em,0)
            (R) -- ++(-1.5em,0)
            (R) -- ($(A) + (1em,0)$)
            (R) -- ($(A) + (-1em,0)$)
        ;

        \node[below=4em of A] (B) {}
                child { node {} child child child child }
                child { node {} child child child child }
                child { node {} child child child child }
                child { node {} child child child child }
            ;
        \draw[mu1,dotted] (B) -- ++(0,1em);

        \node[right=2.5em of B] (B1) {}
                child { node {} child child child child }
                child { node {} child child child child }
                child { node {} child child child child }
                child { node {} child child child child }
            ;
        \draw[mu1,dotted] (B1) -- ++(0,1em);

        \node[left=2.5em of B] (B2) {}
                child { node {} child child child child }
                child { node {} child child child child }
                child { node {} child child child child }
                child { node {} child child child child }
            ;
        \draw[mu1,dotted] (B2) -- ++(0,1em);
            
        \coordinate (AB) at ($(A)!.66!(B)$) ;
        \draw[cutline] ($(AB-|B2)+(-1em,0)$) --
        ($(AB -| B1)+(1em,0)$);

        \node[below=4em of B] (C) {}
                child { node {} child child child child }
                child { node {} child child child child }
                child { node {} child child child child }
                child { node {} child child child child }
            ;
        \draw[mu1,dotted] (C) -- ++(0,1em);

        \node[right=2.5em of C] (C1) {}
                child { node {} child child child child }
                child { node {} child child child child }
                child { node {} child child child child }
                child { node {} child child child child }
            ;
        \draw[mu1,dotted] (C1) -- ++(0,1em);
            
        \node[right=2.5em of C1] (C3) {}
                child { node {} child child child child }
                child { node {} child child child child }
                child { node {} child child child child }
                child { node (Cright) {} child child child child }
            ;
        \draw[mu1,dotted] (C3) -- ++(0,1em);

        \node[left=2.5em of C] (C2) {}
                child { node {} child child child child }
                child { node {} child child child child }
                child { node {} child child child child }
                child { node {} child child child child }
            ;
        \draw[mu1,dotted] (C2) -- ++(0,1em);

        \node[left=2.5em of C2] (C4) {}
                child { node (Cleft) {} child child child child }
                child { node {} child child child child }
                child { node {} child child child child }
                child { node {} child child child child }
            ;
        \draw[mu1,dotted] (C4) -- ++(0,1em);
            
        \coordinate (BC) at ($(B)!.66!(C)$) ;
        \draw[cutline] ($(BC-|Cleft)+(-1em,0)$) --
        ($(BC -| Cright)+(1em,0)$);

        \coordinate (righttree) at ($(Cright) + (2em,0)$);
        \coordinate (lefttree) at ($(Cleft) + (-2em,0)$);
        \coordinate (Cmid) at ($(Cleft)!0.5!(Cright)$);

        \draw[decorate,decoration=brace] (righttree |- R) 
        --
         node[text node, right, xshift=1em] {
                 $T_{\ell^n}(P,P)=1$.
         }
         ($(righttree |- B)+(0,1em)$)
         ;
         \draw[dotted] (lefttree |- B) -- (B2) -- (B) -- (B1) --
         (righttree |- B) -- ++(2em, 0) node[text node, right] {\quad \itshape second stability level};

         \draw[decorate,decoration=brace] (righttree |- B)
        --
         node[text node, right, xshift=1em,yshift=-1ex] {
             \parbox{12em}{
             $T_{\ell^n}(P,P)$ has order $\ell^k$,\\
             with $k=2n-\nu_{\fkl,A}(\pi^r)$.
             }
         }
         ($(righttree |- C)+(0,1em)$)
         ;
         \draw[dotted] (lefttree |- C) --
         (C -| righttree) -- ++(2em,0) node[text node, right] {\quad \itshape first stability level};

         \coordinate[below=4em of Cmid] (D);
        \coordinate (CD) at ($(C)!.6!(D)$) ;
        \draw[cutline] ($(CD-|Cleft)+(-1em,0)$) --
        ($(CD -| Cright)+(1em,0)$);

         \foreach\x in {-9.5,-8.5,...,9.5}
         \draw[mu1,dotted] (.75*\x em, 0 |- D) node {} -- ++(0,1em);

         \draw[decorate,decoration=brace] (righttree |- C)
        --
         node[text node, right, xshift=1em] {
                 $T_{\ell^n}(P,P)$ has order $\ell^n$.
         }
         (righttree |- D)
         ;
         \draw[dotted] (lefttree |- D) -- 
         (righttree |- D) -- ++(2em,0) node[text node, right] {\quad \itshape floor};
    \end{tikzpicture}
    \caption{\label{fig:stability}Stability levels in the $\fkl$-graph}
\end{center}
\end{figure}

We now show that from a computational point of view, we can use the Tate
pairing to orient ourselves in the $\fkl$-isogeny graph. More precisely,
cyclic subgroups of the $\fkl$-torsion with degenerate self-pairing
correspond to kernels of ascending and horizontal isogenies, while
subgroups with non-degenerate self pairing are kernels of descending
isogenies. Before proving this result, we need the following lemma. 

\begin{lemma}
    \label{lemma:atmost2}
    If $k_{\fkl,A}>0$, then there are at most two subgroups of order
    $\ell$ in $A[\fkl^n]$ such that points in these subgroups have
    degenerate self-pairing.
\end{lemma}
\begin{proof}
    We use the shorthand notation $\lambda_{U,V}=\log(T_{\ell^n}(U,V))$
    for $U,V$ any two $\fkl^{n}$-torsion points, and where log is a
    discrete logarithm function in $\mu_{\ell^n}$.

    Suppose that $P$ and $Q$ are two linearly independent
    $\fkl^{n}$-torsion points.  Since all $\fkl^n$-torsion points $R$ can
    be expressed as $R=aP+bQ$, bilinearity of the $\ell^n$-Tate pairing
    gives
%\vspace{-0.2 cm}
%\todo{Je corrige ta formule, il y a un 2 en trop.}
    \begin{align*}
    \lambda_{R,R}&=a^2\lambda_{P,P}+ab\,(\lambda_{P,Q}+\lambda_{Q,P})+b^2\lambda_{Q,Q}
    \pmod{\ell^n},
    \end{align*}
    We now claim that the polynomial
    %\todo{Je change son nom, $P$ ça clashait trop fort.}
    \begin{eqnarray}\label{PairingPolynomial}
        S(a,b)=a^2\lambda_{P,P}+
            ab\,(\lambda_{P,Q}+\lambda_{Q,P})+b^2\lambda_{Q,Q}
    \end{eqnarray}
    is identically zero modulo $\ell^{n-k_{\fkl,A}-1}$ and nonzero modulo
    $\ell^{n-k_{\fkl,A}}$. Indeed, if it were identically zero modulo
    $\ell^{n-k}$, with $k<k_{\fkl,A}$, then we would have
    $T_{\ell^n}(R,R)\in \mu_{\ell^{k}}$, which contradicts the
    definition of $k_{\fkl,A}$. If it were different from zero modulo
    $\ell^{n-k_{\fkl,A}-1}$, then there would be $R\in A[\fkl^n]$ such
    that $T_{\ell^n}(R,R)$ is an $\ell^{k_{\fkl,A}+1}$-th primitive root
    of unity, again contradicting the definition of $k_{\fkl,A}$.

    Points with degenerate self-pairing are roots of $L$. Hence there are
    at most two subgroups of order $\ell$ with degenerate self-pairing. \qed
\end{proof}

In the remainder of this paper, we define by 
\begin{eqnarray*}
    S_{\fkl,A}(a,b)=a^2\lambda_{P,P}+ab(\lambda_{P,Q}+\lambda_{Q,P})+b^2\lambda_{Q,Q}
\end{eqnarray*}
any polynomial defined by a basis $\{P,Q\}$ of $A[\fkl^n]$ in a manner
similar to the proof of Lemma~\ref{lemma:atmost2}, and using the same
notation $\lambda$. 
%\todo{On ne se servira de $S_{\fkl,J}(a,b)$ que bien plus tard (voire pas). Est-il utile de le définir ? À voir\ldots S: On s'en sert pour calculer les points avec des couplages dégénerés - c'est utile pour calculer des isogénies montantes, même si pour notre application on n'a pas besoin}

%\todo{J'ai encore un problème (désolé). Il faut définir un $r$ qui soit
%    une caractéristique du graphe qu'on calcule, pas du point depuis
%    lequel on part. En demandant que la $\fkl$-torsion soit rationnelle,
%    on demande trop. Si on est tout en bas, ce n'est pas ça qu'il nous
%    faut. J'ai l'impression que ce qu'on veut c'est qu'il existe un
%sous-groupe Weil-isotrope.}

\begin{theorem}
    \label{th:pairing-descending} Let $A$ be an p.p.a.v. defined over a finite field $\F_q$ and having locally maximal real multiplication at $\fkl$. Let $P$ be an $\fkl$-torsion point and let $r$ be the smallest integer such that $A[\fkl]\subset A(\F_{q^r})$. Let $n$ be the largest integer such that $A[\fkl^n]\subset A[\F_{q^r}]$. Assume that $k_{\fkl,A}>0$. Consider $G$ a subgroup of $A[\fkl^n]$ such that $\ell^{n-1}G$ is the subgroup generated by $P$. Then the isogeny of kernel $P$ is descending if and only if the Tate pairing is non-degenerate on $G$. It is horizontal or ascending otherwise.
\end{theorem}

%\todo{Les notations $P_1$, $J'$, et $J''$ ici sont d'assez mauvais goût.
%Mais bon...}

\begin{proof}
    We assume $n>1$ and that $k_{\fkl,A}>1$. Otherwise, we consider $A$
    defined over and extension field of $\F_{q^r}$ and apply~\cite[Lemma
    6]{Ionica}. Let $I:A\rightarrow A'$ the isogeny of kernel generated
    by $P$.

    Assume first that $P$ has non-degenerate self-pairing. Let
    $\bar{P}\in G$ such that $\ell^{n-1}\bar P=P$. Then by~\cite[Lemma
    16.2c]{MilneAV} and Lemma~\ref{polarization}, we have
    \begin{eqnarray*}
        T_{\ell^{n-1}}(I(\bar{P}),\alpha(I(\bar{P})))\in \mu_{\ell^{k_{\fkl,A}-1}}\backslash \mu_{\ell^{k_{\fkl,A}-2}},
    \end{eqnarray*}
    where $\alpha$ is a generator of the principal ideal $\fkl'$ such
    that $\fkl\fkl'=\ell \cO_{K_0}$. Since
    $\cO_{K_0}/\alpha\cO_{K_0}\simeq \Z/\ell \Z$, then for any $R\in
    A'[\fkl^n]$, we have $\alpha(R)=\lambda R$, for some $\lambda \in
    \Z/\ell \Z$.  Hence we have
    \begin{eqnarray*}
        T_{\ell^{n-1}}(I(\bar{P}),I(\bar{P}))\in \mu_{\ell^{k_{\fkl,A}-1}}\backslash \mu_{\ell^{k_{\fkl,A}-2}},
    \end{eqnarray*}
    There are two possibilities. Either $A'[\fkl^n]$ is not defined over
    $\F_{q^r}$, or $A'[\fkl^n]$ is defined over $\F_{q^r}$.  In the first
    case, we have $\nu_{\fkl,A'}(\pi^r)< \nu_{\fkl,A}(\pi^r)$ and the
    isogeny is descending.

    Assume now that $A'[\fkl^n]$ is defined over $\F_{q^r}$. Then let
    $P_1$ such that $I(\bar{P})=\ell P_1$. Then
    \begin{eqnarray*}
        T_{\ell^{n}}(P_1,P_1))\in \mu_{\ell^{k_{\fkl,A}+1}}\backslash \mu_{\ell^{k_{\fkl,A}}}.
    \end{eqnarray*}
    By using Proposition~\ref{Equality}, it follows that
    $\nu_{\fkl,A'}(\pi^r)< \nu_{\fkl,A}(\pi^r)$.
    Hence the isogeny is descending.

    Suppose now that the point $P$ has degenerate self-pairing and that
    the isogeny $I$ is descending. Since there are at most 2 points in
    $A[\fkl^n]$ with  degenerate self-pairing, there is at least one
    point in $A[\fkl^n]$ with non-degenerate self-pairing. This point,
    that we denote by $Q$, generates the kernel of a descending isogeny
    $I':A\rightarrow A''$ such that $\End(A')\simeq \End(A'')$. We assume
    first that $A'[\fkl^n]$ and $A''[\fkl^n]$ are not defined over
    $\F_{q^r}$. Then we have
    \begin{eqnarray*}
        T_{\ell^{n-1}}(I(\bar{P}),I(\bar{P})))\in
            \mu_{\ell^{k_{\fkl,A}-2}},
        & T_{\ell^{n-1}}(\ell I(\bar{Q}),\ell (I(\bar{Q})))\in
            \mu_{\ell^{k_{\fkl,A}-3}}\\
        T_{\ell^{n-1}}(\ell I'(\bar{P}),\ell I'(\bar{P}))\in
            \mu_{\ell^{k_{\fkl,A}-4}},
        & T_{\ell^{n-1}}(I'(\bar{Q}),I'(\bar{Q})))\in
            \mu_{\ell^{k_{\fkl,A}-1}}\backslash \mu_{\ell^{k_{\fkl}-2}}\\
    \end{eqnarray*}
    Hence $k_{\fkl,A'}\neq k_{\fkl,A''}$, which is a contradiction. The
    case where  $A'[\fkl^n]$ and $A''[\fkl^n]$ are defined over
    $\F_{q^r}$ is similar. \qed
\end{proof}

\section{Endomorphism ring computation - a depth-first algorithm}
\label{sec:algorithm}

We keep the same setting and notations. In particular, $\ell$ is a fixed
odd prime, and we assume that $\ell \cO_{K_0}=\fkl_1\fkl_2$. We take $J$ to be the Jacobian of a genus 2 curve defined over $\F_q$, which will allow us to compute the Tate pairing efficiently (as explained in Section~\ref{subsec:TatePairing}). We intend to compute the endomorphism ring of $J$, with prior knowledge of the Zeta function of $J$, and the fact that $\End(J)_{\ell}$ contains $\cO_{K_0,\ell}$. We note that this property holds trivially in the case where $\bZ[\pi,\bar\pi]_{\ell}$ contains $\cO_{K_0,\ell}$, although this is not a necessary condition for the algorithm here to work.

\subsection{Description of the algorithm}

A consequence of Proposition~\ref{LLisogenies} is that there are at most
$(\ell+1)(\ell+1)$ rational $\ell$-isogenies preserving the real
multiplication. Since we can compute $\ell$-isogenies over finite
fields~\cite{CosRob,AVISOGENIES}, we use this result to give an algorithm
for computing $\nu_{\fkl,J}(\pi)$, and determine endomorphism
rings locally at $\ell$, by placing them properly in the order lattice as represented in Figure~\ref{fig:order-lattice}.

We define $u_i$ to be the smallest integer such that $\pi^{u_i}-1\in
\fkl_i\cO_K$, and $u$ the smallest integer such that $\pi^{u}-1
\in\ell\cO_K$ (we have $u=\lcm(u_1,u_2)$).  The value of $u$ depends
naturally on the splitting of $\ell$ in $K$ (see~\cite[Prop.
6.2]{FreLau}).  As the algorithm proceeds, the walk on the isogeny graph
considers Jacobians over the extension field $\bF_{p^u}$.

% Je suis pas sûr que ce soit juste: si J[\fkl_1] est défini sur F_q^2 et
% J[\fkl_2] est défini sur F_q^3, alors il ne me semble pas que 
% Note that if we only need to compute the $\fkl_j$-adic valuation of
% $\pi-\bar{\pi}$, for a fixed $j$, we may work over an extension field
% smaller than $u$. 

\paragraph{Idea of the algorithm.}
As noticed by Lemma~\ref{lemma:ordlat-2dim}, we can achieve our goal by considering \emph{separately} the position of the endomorphism ring within the order
lattice with respect to $\fkl_1$ first, and then with respect to
$\fkl_2$. The algorithm below is in effect run twice.

Each move in the isogeny graph corresponds to taking an
$\ell$-isogeny, which is a computationally accessible object. In
our prospect to understand the position of the endomorphism ring with
respect to $\fkl_1$ in Figure~\ref{fig:order-lattice}, we shall not
consider what happens with respect to $\fkl_2$, and vice-versa. Our input
for computing an $\ell$-isogeny is a Weil-isotropic kernel.
Because we are interested in isogenies preserving the real
multiplication, this entails that we consider kernels of the form
$K_1+K_2$, with $K_i$, $i=\overline{1,2}$, a cyclic subgroup of $J[\fkl_i]$. By
Proposition~\ref{prop:weil-isotropic-l1l2}, such a group is
Weil-isotropic. There are up to $(\ell+1)^2$ such subgroups.

Let $\fkl$ be either $\fkl_1$ or $\fkl_2$.  The algorithm computes
$\nu_{\fkl,J}(\pi)$ in two stages. 

Our algorithm stops when the floor of rationality has been hit in $\fkl$,
i.e. the only rational cyclic group in $J[\fkl]$ is the one generating
the kernel of the ascending $\fkl$-isogeny. If $(u,\ell)=1$, one may
prove that testing rationality for the isogenies is equivalent to
$J[\fkl]\subset J(\F_{q^u})$.  Otherwise, in order to test rationality
for the isogeny at each step in the algorithm, one has to check whether
the kernel of the isogeny is $\F_q$-rational.

\paragraph{Step 1.} The idea is to walk the isogeny graph until we reach
a Jacobian which is on the second stability level or below (which might
already be the case, in which case we proceed to Step 2).  If the
Jacobian $J$ is above the second stability level, we need to construct
several chains of $\ell$-isogenies, not backtracking with respect
to $\fkl$, to make sure at least one of them is descending in the
$\fkl$-direction.  This proceeds exactly as in~\cite{FouMor}. The number
of chains depends on the number of horizontal isogenies and thus on the
splitting of $\fkl$ in $K$ (due to the action of the Shimura class
group). If $\fkl$ is split, one  needs three isogeny chains to ensure
that one path is descending.

If an isogeny in the chain is descending, then the path continues
descending, assuming the isogeny walk does not backtrack with respect to
$\fkl$ (this aspect is discussed further below).  We are done
constructing a chain when we have reached the second stability level for
$\fkl$, which can be checked by computing self-pairing of appropriate
$\ell^n$-torsion points. The length of the shortest path gives the
correct level difference between the second stability level and the
Jacobian~$J$. The pseudocode for this step is given in Algorithm~\ref{LocallyMaximal}.

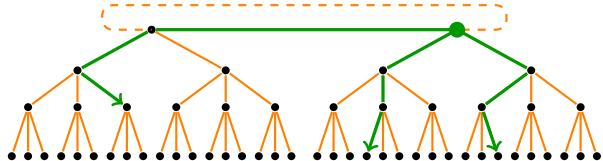
\begin{figure}
    \begin{center}
        \begin{tikzpicture}[
            point/.style={circle, fill=black, minimum size=3pt, inner sep=0pt},
            every node/.style={point},
            mu1/.style={orange, thick},
            mu2/.style={violet, thick},
            every child/.style={mu1},
            level 1/.style={sibling distance=6em, level distance=1.66em},
            level 2/.style={sibling distance=2em, level distance=1.5em},
            level 3/.style={sibling distance=0.66em, level distance=2em},
            skip loop/.style={rounded corners, to path={-- ++(2em,0) --
            ++(0,#1) -| ($(\tikztotarget)+(-2em,0)$) -- ++(2em,0)}},
        ]
        \path node (A) {}
                child { node (A1) {}
                    child { node (A11) {}
                        child { node (A111) {} }
                        child { node (A112) {} }
                        child { node (A113) {} }
                    }
                    child { node (A12) {}
                        child { node (A121) {} }
                        child { node (A122) {} }
                        child { node (A123) {} }
                    }
                    child { node (A13) {}
                        child { node (A131) {} }
                        child { node (A132) {} }
                        child { node (A133) {} }
                    }
                }
                child { node (A2) {}
                    child { node (A21) {}
                        child { node (A211) {} }
                        child { node (A212) {} }
                        child { node (A213) {} }
                    }
                    child { node (A22) {}
                        child { node (A221) {} }
                        child { node (A222) {} }
                        child { node (A223) {} }
                    }
                    child { node (A23) {}
                        child { node (A231) {} }
                        child { node (A232) {} }
                        child { node (A233) {} }
                    }
                }
                ;
        \node[right=12em of A] (B) {};
        \path (B)
                child { node (B1) {}
                    child { node (B11) {}
                        child { node (B111) {} }
                        child { node (B112) {} }
                        child { node (B113) {} }
                    }
                    child { node (B12) {}
                        child { node (B121) {} }
                        child { node (B122) {} }
                        child { node (B123) {} }
                    }
                    child { node (B13) {}
                        child { node (B131) {} }
                        child { node (B132) {} }
                        child { node (B133) {} }
                    }
                }
                child { node (B2) {}
                    child { node (B21) {}
                        child { node (B211) {} }
                        child { node (B212) {} }
                        child { node (B213) {} }
                    }
                    child { node (B22) {}
                        child { node (B221) {} }
                        child { node (B222) {} }
                        child { node (B223) {} }
                    }
                    child { node (B23) {}
                        child { node (B231) {} }
                        child { node (B232) {} }
                        child { node (B233) {} }
                    }
                }
                ;

%         \node[right=12em of B] (C) {};
%                 \begin{scope}[opacity=0.5,scale=0.5]
%         \path (C)
%  child { node {}
%  child { node {} child { node {} } child { node {} } child { node {} } }
%  child { node {} child { node {} } child { node {} } child { node {} } }
%  child { node {} child { node {} } child { node {} } child { node {} } }
%  }
%  child { node {}
%  child { node {} child { node {} } child { node {} } child { node {} } }
%  child { node {} child { node {} } child { node {} } child { node {} } }
%  child { node {} child { node {} } child { node {} } child { node {} } }
%  }
%                 ;
%                 \end{scope}
                \draw[mu1]
                (A) -- (B) edge[skip loop=1em, dashed] (A);
            \draw[very thick, green!60!black]
            (B) -- (A) -- (A1) edge[->] (A13)
            (B) -- (B2) -- (B21) edge[->] (B213)
            (B) -- (B1) -- (B12) edge[->] (B121);
            \node[minimum size=6pt, fill=green!60!black] at (B) {};
        \end{tikzpicture}
        \caption{\label{fig:3paths}At least one in three non-backtracking
            paths has minimum distance
        to a given level.}
    \end{center}
\end{figure}

Figure~\ref{fig:3paths} represents for $\ell=3$ a situation where only
three non-backtracking paths can guarantee that at least one of them is
consistently descending.

% the most difficult case for $\ell=3$,  the vertices inside the black line square correspond to orders containing $\Z[\pi,\bar{\pi}]$. The red line corresponds to the second stability level with respect to $\fkl$. A descending path of $(\ell, \ell)$-isogenies from the top to the second stability level is shown. 

\paragraph{Step 2.} We now assume that $J$ is on the second stability level or
below, with respect to $\fkl$. We construct a non-backtracking path of
$\ell$-isogenies, which are consistently descending with respect
to $\fkl$. In virtue of Theorem~\ref{th:pairing-descending}, this can be
achieved by picking Weil-isotropic kernels whose $\fkl$-part (which is
cyclic) correspond to a non-degenerate self-pairing $T_{\ell^n}(P,P)$.
We stop when we have reached the floor of rationality in $\fkl$, at which
point the valuation $\nu_{\fkl,J}(\pi)$ is obtained.

Note that at each step taken in the graph, if $J[\fkl']$ (where $\fkl'$
is the other ideal) is not rational, then we ascend in the
$\fkl'$-direction, in order to compute an $\ell$-isogeny. As said
above, this has no impact on the consideration of what happens with
respect to $\fkl$. This step is summarized in Algorithm~\ref{LocallyMaximal2}.

\paragraph{Ensuring isogeny walks are not backtracking}
\label{ensuring:nbt}
As said above, ensuring that the isogeny walk in Step 2 is not
backtracking is essentially guaranteed by
Theorem~\ref{th:pairing-descending}. Things are more subtle for Step~1.
Let $J_1$ be a starting Jacobian, and $I: J_1\rightarrow J_2$ an
$\ell$-isogeny whose kernel is $V\subset J[\ell]$.  Recall that
there are at most $(\ell+1)^2$ Weil-isotropic kernels of the form
$K_1+K_2$ within $J_2[\fkl_1]+J_2[\fkl_2]$ for candidate isogenies
$I':J_2\rightarrow J_1$. All such isogenies whose kernel has the same
component on $J_2[\fkl_1]$ as the dual isogeny $\hat I$ are backtracking
with respect to $\fkl_1$ in the isogeny graph. One must therefore
identify the dual isogeny $\hat I$ and its kernel.  Since $\hat I$ is
such that $\hat I\circ I=[\ell]$, we have that $\Ker \hat I =
I(J_1[\ell])$. If computing $I(J_1[\ell])$ is possible\footnote{Computing
isogenous Jacobians by isogenies is easier than computing images of
divisors.  The \texttt{avisogenies} software~\cite{AVISOGENIES} performs
the former since its inception, and the latter in its development
version, as of 2014.}, this solves the issue. If not, then enumerating
all possible kernels until the dual isogeny is identified is possible,
albeit slower.

\begin{algorithm}[h!]
\caption{Computing the endomorphism ring: Step 1}\label{LocallyMaximal}
\begin{algorithmic}[1]
\REQUIRE A Jacobian $J$ of a genus-2 curve defined over $\F_q$ with CM by a field $K$, and $\alpha$ such that $\fkl=\alpha\cO_K$ divides $\ell\cO_K$. We require that $J$ is above the second stability level with respect to $\fkl$.
\ENSURE A Jacobian $J'$ on or below the second stability level with
respect to $\fkl$, and the distance from $J$ to this Jacobian.
\STATE Let $\pi$ be the Frobenius endomorphism of $J$ and $u$ be the
    smallest integer s.t. $\pi^u-1 \equiv 0 \pmod {\ell\cO_K}$.
\STATE Compute a basis of $J[\ell^{\infty}](\F_{q^u})$.
\STATE Compute $a,b\in\bQ$ such that $\alpha=a+b(\pi+\bar{\pi})$.
\STATE Let $n$ be the largest integer such that $J[\fkl^{n}]\subset J(\F_{q^u})$.  
\STATE $J_{1}\leftarrow J$, $J_{2}\leftarrow J$, $J_{3}\leftarrow J$.
\STATE $\kappa_{1}\leftarrow\{0\}$, $\kappa_{2}\leftarrow\{0\}$, $\kappa_{3}\leftarrow\{0\}$.
\STATE $\mathrm{length}\leftarrow 0$.
\WHILE{\textbf{true}} 
\STATE $\mathrm{length}\leftarrow \mathrm{length}+1$.
\FORALL {i=1,2,3}
\STATE Compute the matrix of $\pi$ in $J_i[\ell^{\infty}](\bF_{q^u})$.
\STATE Compute bases for $J_i[\fkl](\F_{q^u})$ and $J_i[\fkl'](\F_{q^u})$ using
$\alpha=a+b(\pi+\bar{\pi})$.
\STATE Pick at random $P_i\in J_i[\fkl](\F_{q^u})$ such that $P_i\notin\kappa_i$.
\STATE Pick at random $P'_i\in J_i[{\fkl'}](\F_{q^u})$.
\STATE Compute the $\ell$-isogeny $I: J_i\rightarrow J'_i=J_i/\langle
P_i,P'_i\rangle$.
\STATE $\kappa_i\leftarrow I(J[\fkl])$; $J_i\leftarrow J'_i$.
\STATE Compute $S_{\fkl,J}$.
\IF {$S_{\fkl,J}\not=0$}
\STATE \textbf{return} $\mathrm{length}$.
\ENDIF
\ENDFOR
\ENDWHILE
\end{algorithmic}
\end{algorithm}
%stopzone

\begin{algorithm}
\caption{Computing the endomorphism ring: Step 2}\label{LocallyMaximal2}
\begin{algorithmic}[1]
\REQUIRE A Jacobian $J$ of a genus-2 curve defined over $\F_q$ with CM by a field, and $\alpha$ such that $\fkl=\alpha\cO_K$ divides $\ell\cO_K$. We require that $J$ is on or below the second stability level with respect to $\fkl$ (see Algorithm~\ref{LocallyMaximal}).
\ENSURE The $\fkl$-distance from $J$ to the floor.  
\STATE $\mathrm{length}\leftarrow 0$.
\WHILE{\textbf{true}} 
\STATE Let $\pi$ be the Frobenius of $J$ and let $u$ the smallest integer s.t. $\pi^u-1 \equiv 0 \pmod {\ell\cO_K}$ and compute a basis of $J[\ell^{\infty}](\F_{q^u})$.
\STATE Let $n$ the largest integer such that $J[\fkl^{n}]\subset J(\F_{q^u})$.  
\IF {$n=0$}
\STATE \textbf{return} $\mathrm{length}$.
\ENDIF
\STATE Compute the matrix of $\pi$ in $J_i[\ell^{\infty}](\bF_{q^u})$.
\STATE Let $\fkl'=\ell/\fkl$. Compute bases for $J_i[\fkl](\F_{q^u})$ and $J_i[\fkl'](\F_{q^u})$ %using $\alpha=a+b(\pi+\bar{\pi})$.
\STATE Consider $P_1,P_2$ a basis of $J[\fkl^{n}](\F_{q^u})$
\STATE Compute $S_{\fkl,J}$ and take $x_1,x_2 \in \F_{\ell}$ such that $S_{\fkl,J}(x_1,x_2)\neq 0$. 
\STATE $P\leftarrow \ell^{n-1}(x_1P_1+x_1P_2)$.
\STATE Pick at random $P'_i\in J_i[{\fkl'}](\F_{q^u})$.
\STATE Compute the $\ell$-isogeny $I:J'\leftarrow J/\langle P, P'\rangle$
\STATE $J\leftarrow J'$.
\STATE $\mathrm{length}\leftarrow \mathrm{length}+1$.
\ENDWHILE
\end{algorithmic}
\end{algorithm}

\subsection{Complexity analysis}
\label{subsec:EisentragerLauter}
In this Section, we give a complexity analysis of
Algorithms~\ref{LocallyMaximal} and~\ref{LocallyMaximal2} and compare their performance to that of the Eisentr\"ager-Lauter algorithm for computing the endomorphism ring locally at $\ell$, for small $\ell$. If $\ell$ is large, one should use Bisson's algorithm~\cite{BissonGenus2}. Computing a
bound on $\ell$ for which one should switch between the two algorithms
and a full complexity analysis of the algorithm for determining the
endomorphism ring completely is beyond the scope of this paper.

% \vspace{0.3 cm}
\paragraph{The Eisentr\"ager-Lauter algorithm}
For completeness, we briefly recall the Eisentr\"ager-Lauter
algorithm~\cite{EisLau}.  For a fixed order $\cO$ in the lattice of
orders of $K$, the algorithm tests whether $\cO\subset\End(J)$.  This is
done by computing a $\Z$-basis of \cO\ and checking whether its elements
are endomorphisms of $J$ or not.  In order to test if $\alpha \in \cO$ is
an endomorphism, we write
$$\alpha=\frac{a_0+a_1\pi+a_2\pi^2+a_3\pi^3}{N},$$ with $a_i$ integers
whose greatest common divisor is coprime to $N$ ($N$ is the smallest
integer such that $N\alpha \in \Z[\pi]$).  Using \cite[Prop. 7]{EisLau},
we get $\alpha\in \End(J)$ if and only if $\sum_i a_i\pi^i$ acts as zero
on the $N$-torsion.

Freeman and Lauter~\cite{FreLau} work locally modulo prime divisors of
$N$. For all orders such that $\bZ[\pi]\subset\cO\subset\cO_K$, the
denominators $N$ considered are divisors of $[\cO_K:\Z[\pi,\bar\pi]]$
(see~\cite[Lemma 3.3 and Corollary 3.6]{FreLau}).
Moreover, Freeman and Lauter show that if $N$ factors as
$\ell_1^{d_1}\ell_2^{d_2}\ldots \ell_r^{d_r}$, it suffices to check if
$$\frac{a_0+a_1\pi+a_2\pi^2+a_3\pi^3}{\ell_i^{d_i}},$$ 
is an endomorphism, for all $i$.  The
advantage of working locally is that instead of working over the extension field generated by the coordinates of the $N$-torsion points, we may work over the field of definition of the $\ell_i^{d_i}$-torsion, for every prime factor $\ell_i$ separately. Nevertheless, it should be noted that the exponent $d_i$ can be as large as the $\ell_i$-valuation of the conductor $[\cO_K:\bZ[\pi,\bar{\pi}]]$.

\bigskip
We now set some notations for giving the complexity of algorithms from
Section~\ref{sec:algorithm} as well as that of the Eisentr\"ager-Lauter algorithm. We consider the complexity for one odd prime $\ell$ dividing
$[\cO_K:\Z[\pi,\bar\pi]]$, and assume that $(\ell,p)=1$. Following the
notation in Section~\ref{GraphStructure}, we denote
$h_i=\nu_{\fkl_i,\cO_K}(\pi)$ for $i=1,2$. It follows that
$\nu_\ell([\cO_K:\cO_{K_0}[\pi,\bar\pi]])=h_1+h_2$. The order
$\bZ[\pi,\bar\pi]$ might be smaller than $\cO_{K_0}[\pi,\bar{\pi}]$, thus
we denote $h_0=\nu_\ell([\cO_{K_0}[\pi,\bar{\pi}]:\bZ[\pi,\bar\pi]])$.
Note though that for most practical uses of our algorithm, we expect to
gain knowledge that $\End(J)$ has maximal real multiplication from the
fact that $\bZ[\pi,\bar\pi]$ is an $\cO_{K_0}$-order itself, which
implies $h_0=0$. It makes sense to neglect $h_0$ in this case.
% The Eisenträger-Lauter algorithm is sensible to
% $\nu_\ell([\cO_K:\bZ[\pi,\bar\pi])=h_0+h_1+h_2$. 
Finally, we let as before $u$ be the smallest integer such that
$\pi^u\equiv1\mod\ell\cO_K$, so that the $\ell$-torsion on $J$ is defined
over $\F_{q^u}$. According to~\cite[Prop. 6.2]{FreLau}, we have $u\in
O(\ell^2)$ since $\ell$ splits in $K_0$.

We now give the complexity of the algorithm from Section~\ref{sec:algorithm}. First we compute a basis of the ``$\ell^\infty$-torsion over $\F_{q^u}$'', i.e.\ the $\ell$-Sylow subgroup of $J(\F_{q^u})$, which corresponds to $J[\ell^n](\F_{q^u})$ for some integer $n$. We assume that the zeta function of $J$ and the
factorization of $\#J(\F_{q^u})=\ell^sm$ are given. We denote by $M(u)$
the number of a multiplications in $\F_q$ needed to perform one
multiplication in the extension field of degree $u$. The computation of
the Sylow subgroup basis costs $O(M(u)(u\log q+n\ell^2))$ operations in
$\F_q$, as described in~\cite[§3]{BiCoRo14}.   

Then we compute the matrix of the Frobenius on the $\ell$-torsion. Using
this matrix, we write down the matrices of $\alpha_1$ and $\alpha_2$ in
terms of the the matrix of $\pi+\bar{\pi}$. Finally, computing
$J[\fkl_i]$ for $i=1,2$ is just linear algebra and has negligible cost.
For each $i$, the cost of computing the Tate pairing is related to the
integers $r_i$ and $n_i$ as defined in Proposition~\ref{Equality}. We
bound these by $r_i\leq u$, and $n_i\leq n$. Computing the Tate pairing
thus costs $O(M(u)(n\log \ell+u\log q))$ operations in $\F_q$, where the
first term is the cost of Miller's algorithm and the second one is the
cost for the final exponentiation.

The cost of computing an $\ell$-isogeny using the algorithm of
Cosset and Robert~\cite{CosRob} is $O(M(u)\ell^4)$ operations in $\F_q$.
We conclude that the cost of Algorithms~\ref{LocallyMaximal}
and~\ref{LocallyMaximal2} is $$\text{cost}_{\text{algorithms
1+2}}=O(\max(h_1,h_2)M(u)(u\log q+n\ell^2+\ell^4)).$$

The complexity of Freeman and Lauter's algorithm is dominated by the cost
of computing the $\ell$-Sylow subgroup of the Jacobian defined over the
extension field containing the $\ell^{d}$-torsion, where $d$ is bounded
by
$\nu_\ell([\cO_K:\bZ[\pi]])=\nu_\ell([\cO_K:\bZ[\pi,\bar\pi]])=h_0+h_1+h_2$
(recall that $\ell$ and $\pi$ are coprime). The degree of this extension
field is $u\ell^{d-1}$ by~\cite[Prop. 6.3]{FreLau}.  This leads to
$$\text{cost}_{\text{EL}}=O(M(u\ell^{d-1})(u\ell^{d-1}\log q+(n+d-1)\ell^2)).$$
    
\begin{table}
    \begin{center}
        \begin{tabular}{|c|c|}
            %heading
            \hline
            Freeman and Lauter  & This work (Algorithms~\ref{LocallyMaximal}
            and~\ref{LocallyMaximal2}) \\
            \hline
            {$O(M(u\ell^{d-1})(u\ell^{d-1}\log q+(n+d-1)\ell^2))$}
            &{$O(\max(h_1,h_2)M(u)(u\log q+n\ell^2+\ell^4))$}\\
            \hline
        \end{tabular}
        \medskip
        \caption{\label{Table1} Cost for computing the endomorphism ring
        locally at $\ell$; we have $u=O(\ell^2)$, $d\leq h_0+h_1+h_2$, and $h_0=0$ is a typical condition for this work to apply}.
    \end{center}
\end{table}

\subsection{Practical experiments}
Let $J$ be the Jacobian of the hyperelliptic curve defined by
\begin{align*}
y^2 &= 17422020 + 847562x + 37917221x^2 + 268754x^3 + 4882157x^4 + 14143796x^5 + 50949756x^6
\end{align*} 
over $\F_p$, with $p=53050573$. The curve has complex multiplication by
$\cO_K$, with $K=\Q(\zeta)$, defined by the equation
$\zeta^4+175\zeta^2+6925=0$. A Weil number for this Jacobian, as well as
the corresponding characteristic polynomial, are given as follows:
\begin{gather*}
    \pi= \frac{1}{15}(45\zeta^3 + 422\zeta^2 + 14940\zeta + 79450),\\
    \pi^4-s_1\pi^3+s_2\pi^2-s_1p\pi+p^2=0,\ \text{with}\ s_1=11340,\
    s_2=135934954.
\end{gather*}

The real multiplication subfield $K_0$ has class number~1, and~$\ell=3$
splits in $K_0$ as $3=\alpha_1\alpha_2$.  The corresponding valuations of
the Frobenius are $\nu_{\alpha_1,\cO_K}(\pi)=10$ and
$\nu_{\alpha_2,\cO_K}(\pi)=2$. The analogue to
Figure~\ref{fig:order-lattice} is thus a lattice of 20 possible orders to
choose from in order to determine $\End(J)$.

Our algorithm computes the 3-torsion group, which is defined over
$\F_{p^2}$.  Note that in contrast, the Eisentr\"ager-Lauter algorithm
computes the $3^{10}$-torsion group, defined over $\F_{p^{39366}}$. 

We report experimental results of our implementation, using Magma 2.20-6
and \texttt{avisogenies} 0.6, on a Intel Core i5-4570 CPU with clock
frequency 3.2 GHz.  Our computation of $\End(J)$ with
Algorithms~\ref{LocallyMaximal} and~\ref{LocallyMaximal2} goes as
follows. Computation shows that the Tate pairing is degenerate on
$J[\fkl_1]$. We thus use Algorithm~\ref{LocallyMaximal} to find a
shortest path from $J$, not backtracking with respect to $\fkl_1$, and
reaching a Jacobian on or above the second stability level. This path is
made of $\ell$-isogenies defined over $\F_p$, and computed with
\texttt{avisogenies} from their kernels (here, only what happens with
respect to $\fkl_1$ is interesting). Such a path with length 3 is found
in 20 seconds, where most of the time (15 seconds) is spent on ensuring
that the isogeny walks are non-backtracking (see remark on
page~\pageref{ensuring:nbt}). From there, a consistently descending path
of length $5$ down to the floor is constructed using
Algorithm~\ref{LocallyMaximal2} in 3 seconds. This leads to
$\nu_{\fkl_1,J}(\pi)=8$. As for $\fkl_2$, the Jacobian $J$ is
below the second stability level, so Algorithm~\ref{LocallyMaximal2}
applies, and finds $\nu_{\fkl_2,J}(\pi)=1$ in 1 second. In total,
the computation $\End(J)$ in this example takes 24 seconds.

\section{Conclusion}

We have described the structure of the degree-$\ell$ isogeny graph
between abelian surfaces with maximal real multiplication. From a
computational point of view, we exploited the structure of the graph to
describe an algorithm computing locally at $\ell$ the endomorphism ring
of an abelian surface with maximal real multiplication.

In this work we used the assumption that $K_0$ has class number 1 to give the structure of the lattice of orders with locally maximal real multiplication at $\ell$ and also assumed that the ideal $\fkl$ is trivial in the narrow class group of $K_0$. This allowed us to exhibit an $\fkl$-isogeny graph between principally polarized abelian varieties. 

For a generalization of this work the case where $K_0$ has class number greater the reader is referred to~\cite{Brooks,Martindale}. In particular, the assumption that $\fkl$ is trivial in the narrow class group is left out in~\cite{Brooks}. This leads to an $\fkl$-isogeny graph between polarized abelian varieties, belonging to different polarization classes.

Further research is needed to extend these results to a general setting and compute endomorphism ring in the case where $\ell$ divides $[\cO_{K_0}:\Z[\pi+\bar{\pi}]]$. Our belief is that the right approach to follow is first to determine the real multiplication order $\cO_0$ and secondly to use an algorithm similar to ours, exploiting the structure of the isogeny graph between principally polarized abelian variety with real multiplication by $\cO_0$.

The reader should also note recent results and ongoing work on the computation of $\fkl$-isogenies (via modular polynomials~\cite{Milio,Martindale} and~\cite{Dudeanu}). It would be interesting to compare the performance of algorithms navigating into the $\fkl$-isogeny graphs against that of navigating in $\ell$-isogeny graphs, in order to see whether our methods for computing endomorphism rings can be improved.

\bibliographystyle{plain}
\bibliography{sorina1}

\appendix

\section{Appendix: additional example}

We consider the quartic CM field $K$ with defining equation
$X^4+81X^2+1181$. The real subfield is $K_0=\Q(\sqrt{1837})$, and has
class number~1.  In the real subfield $K_0$, we have
$3=\alpha_1\alpha_2$, with $\alpha_1=\frac{43+\sqrt{1837}}{2}$ and
$\alpha_2$ its conjugate.  We consider a Weil number $\pi$ of relative
norm $85201$ in $\cO_K$. We have that
$\nu_{\alpha_1}(\fkf_{\Z[\pi,\bar\pi]})=2$ and
$\nu_{\alpha_2}(\fkf_{\Z[\pi,\bar\pi]})=1$. Note that $\fkl_1$ is inert
and $\fkl_2$ is split in $K$. Our implementation with Magma produced the
graph in Figure~\ref{fig:LargerGraph}. 

\begin{figure}[h!]
%\begin{minipage}[b]{0.2\linewidth}
%\vspace*{-50 cm}
%\begin{tikzpicture}[
%O1a/.style={circle,fill=red,minimum size=6pt, inner sep=0pt},
%O3a/.style={rectangle,rotate=45,fill=green,minimum size=6pt, inner sep=0pt},
%O3b/.style={rectangle,fill=yellow,minimum size=6pt, inner sep=0pt},
%O9b/.style={regular polygon, regular polygon sides=3, minimum size=0, inner sep=2pt, rotate=-90,fill=blue,minimum size=8pt, inner sep=0pt},
%O9a/.style={circle,fill=pink,minimum size=6pt, inner sep=0pt},
%O27a/.style={rectangle,rotate=45,fill=orange,minimum size=6pt, inner sep=0pt},
%	]
%
%\node[O1a] (X_1_1_1_1) at (0:0) {};
%%$\mathcal{O}_{K}$ at (3,0)
%\node[O3a,below left=20pt of X_1_1_1_1] (X_1_3_1_1) {};
%\node[O3b,below right=20pt of X_1_1_1_1] (X_9_3_2_1) {};
%\node[O9b,below left=20pt of X_1_3_1_1] (X_6_9_1_1) {};
%\node[O9a,below left=17pt of X_9_3_2_1] (X_9_9_1_1)  {};
%\node[O27a,below left=20pt of X_9_9_1_1] (X_9_9_9_1)  {};
%\draw[violet] (X_1_1_1_1)--(X_1_3_1_1);
%\draw[orange] (X_1_1_1_1)--(X_9_3_2_1);
%\draw[orange] (X_1_3_1_1)--(X_9_9_1_1);
%\draw[violet] (X_9_3_2_1)--(X_9_9_1_1);
%\draw[violet](X_1_3_1_1)--(X_6_9_1_1);
%\draw[orange] (X_6_9_1_1)--(X_9_9_9_1);
%\draw[violet] (X_9_9_1_1)--(X_9_9_9_1);
%\end{tikzpicture}
%%\end{center}
%\\
%\vspace*{3 cm}
%\end{minipage}
%%$[A,B]=[ 81, 1181 ]$, $p=85201$, $\ell=3$
%\begin{minipage}[b]{0.5\linewidth}
%\begin{center}
%%\vspace*{-1cm}
%\includegraphics[scale=0.4]{pictures/graph_81_1181_3_85201.pdf}
%%\caption{1-round Feistel scheme}
%\vspace*{-2 cm}
%\end{center}
%\end{minipage}
    \begin{center}
        \includegraphics[scale=0.4,angle=90]{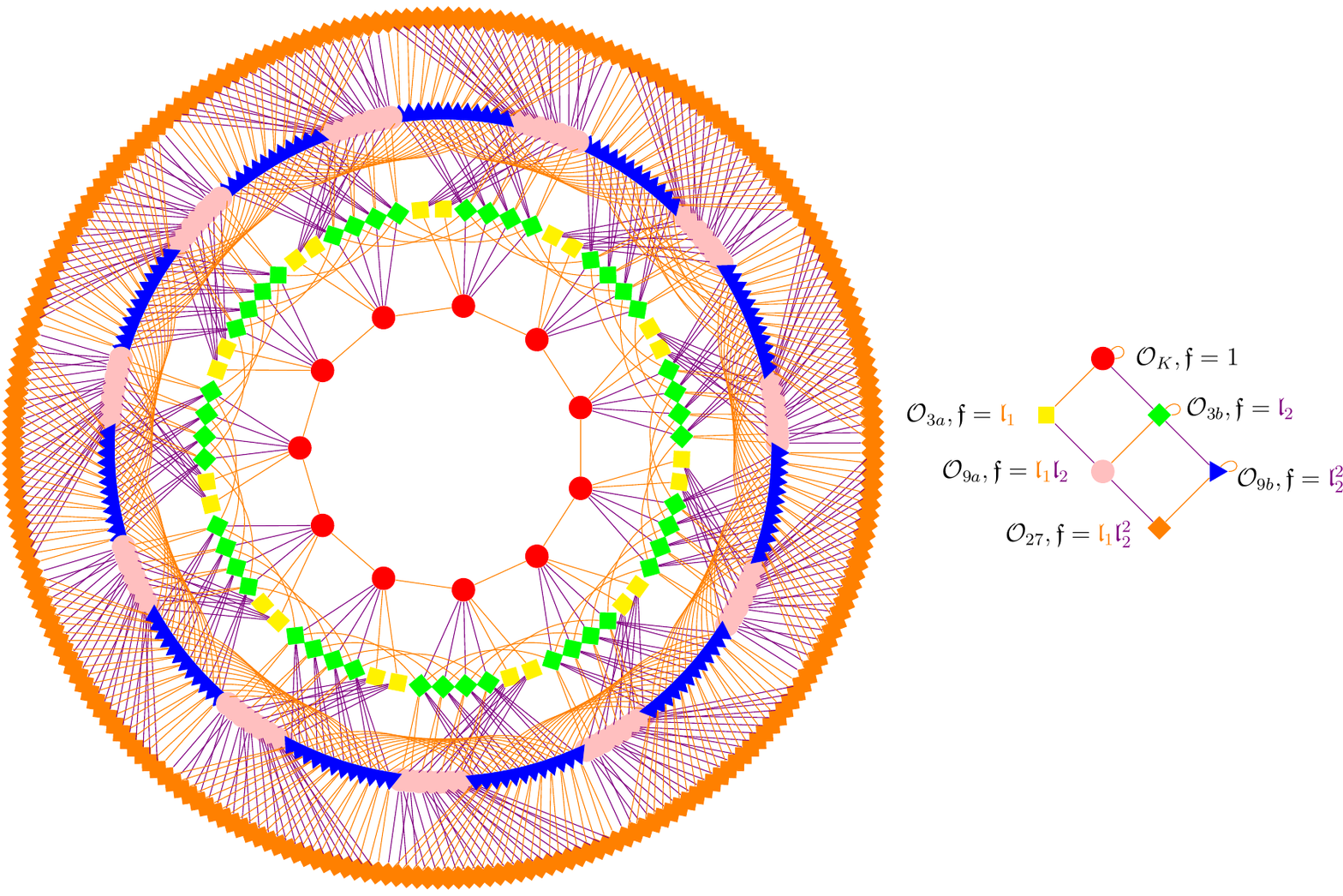}
\caption{\label{fig:LargerGraph} A larger example}
    \end{center}
\end{figure}

%TODO : quelques beaux dessins de graphes. 

\end{document}